\documentclass[12pt,a4paper]{article}
\usepackage{amssymb,amsthm,amsmath}
\usepackage[english]{babel}
\usepackage{fullpage}
\numberwithin{equation}{section}
\date{30 October 2007; v2 15 May 2008}
\title{Noncommutative families of instantons \\[20pt] }
\author{Giovanni Landi$\strut^{1}$, Chiara Pagani$\strut^{2}$,
\\[15pt]
Cesare Reina$\strut^{3}$, Walter D. van Suijlekom$\strut^{4}$
\\[35pt]
$\strut^{1}$ Dipartimento di Matematica e Informatica, Universit\`a
di Trieste \\
Via A.Valerio 12/1, I-34127 Trieste, Italy\\
and INFN, Sezione di Trieste, Trieste, Italy
\\ landi@univ.trieste.it
\\[5pt]
$\strut^{2}$ Department of Mathematical Sciences, University of Copenhagen,
\\ Universitetsparken 5,  DK-2100 Copenhagen, Denmark
\\ pagani@math.ku.dk, \, pagani@sissa.it
\\[5pt]
$\strut^{3}$ International School for Advanced Studies, \\
Via Beirut
2-4, I-34014 Trieste, Italy \\ reina@sissa.it
\\[5pt]
$\strut^{4}$ IMAPP, Radboud Universiteit \\Toernooiveld 1, 6525 ED Nijmegen, The Netherlands\\
waltervs@math.ru.nl
}

\def\ii{\mathrm{i}}
\def\into{\hookrightarrow}

\newcommand{\nn}{\nonumber}
\def\bra#1{\left\langle #1\right|}
\def\ket#1{\left| #1\right\rangle}

\DeclareMathOperator{\Mat}{Mat}

\DeclareMathOperator{\ch}{ch}
\DeclareMathOperator{\tr}{tr}
\DeclareMathOperator{\SU}{SU}
\DeclareMathOperator{\su}{su}
\DeclareMathOperator{\SO}{SO}
\DeclareMathOperator{\Spin}{Spin}
\DeclareMathOperator{\Sp}{SO}
\DeclareMathOperator{\SL}{SL}

\DeclareMathOperator{\U}{U}
\DeclareMathOperator{\diag}{diag}

\newcommand{\uno}{z_1}
\newcommand{\due}{z_2}
\newcommand{\tre}{z_3}
\newcommand{\qu}{z_4}
\newcommand{\unob}{z_1^*}
\newcommand{\dueb}{z_2^*}
\newcommand{\treb}{z_3^*}
\newcommand{\qub}{z_4^*}

\newcommand{\wuno}{w_1}
\newcommand{\wdue}{w_2}
\newcommand{\wtre}{w_3}
\newcommand{\wqu}{w_4}
\newcommand{\wunob}{w_1^*}
\newcommand{\wdueb}{w_2^*}
\newcommand{\wtreb}{w_3^*}
\newcommand{\wqub}{w_4^*}

\newcommand{\ta}{\widetilde{\alpha}}
\newcommand{\tb}{\widetilde{\beta}}
\newcommand{\tx}{\widetilde{x}}

\newcommand{\ba}{a^*}
\newcommand{\bb}{b^*}
\newcommand{\bc}{c^*}
\newcommand{\bd}{d^*}
\newcommand{\bg}{g^*}

\newcommand{\IR}{\ensuremath{\mathbb{R}}}
\newcommand{\IZ}{\ensuremath{\mathbb{Z}}}
\newcommand{\IC}{\ensuremath{\mathbb{C}}}
\newcommand{\IH}{\ensuremath{\mathbb{H}}}

\newcommand{\II}{\ensuremath{\mathbb{I}}}
\newcommand{\IP}{\ensuremath{\mathbb{P}}}
\newcommand{\IT}{\ensuremath{\mathbb{T}}}

\newcommand{\ce}{\ensuremath{\mathcal{E}}}

\def\half{\tfrac{1}{2}}

\newcommand{\ot}{\otimes}

\newcommand{\pot}{\overset{.}{\otimes}}

\newtheorem{teo}{Theorem}[section]
\newtheorem{prop}[teo]{Proposition}
\newtheorem{cor}[teo]{Corollary}
\newtheorem{rem}[teo]{Remark}

\newtheorem{lemma}[teo]{Lemma}

\newcommand{\beq}{\begin{equation}}
\newcommand{\eeq}{\end{equation}}

\renewcommand{\bar}[1]{\overline{#1}}   
\def\class{{(0)}}                       
\def\dd{\mathrm{d}}                     
\def\det{\mathrm{det}}                  
\def\E{\mathcal{E}}                     
\def\cg{\mathfrak{g}}                    
\def\HC{\textup{HC}}
\def\id{\mathrm{id}}                    
\def\M{\mathcal{M}}                     
\def\SL{\mathrm{SL}}                    
\def\S{S_\theta}                        
\def\Sp{\mathrm{Sp}}                    
\def\K{\textup{K}}
\def\KK{\textup{KK}}

\newcommand{\A}{\mathcal{A}}
\newcommand{\B}{\mathcal{B}}
\newcommand{\C}{\mathcal{C}}

\begin{document}

\maketitle

\begin{abstract}
We construct $\theta$-deformations of the classical groups $\SL(2,\IH)$ and $\Sp(2)$.   Coacting on a basic instanton on a noncommutative four-sphere
$S^4_\theta$, we construct a noncommutative family of instantons of charge 1. The family  is parametrized by the quantum quotient of
$\SL_\theta(2,\IH)$ by $\Sp_\theta(2)$.
\end{abstract}

\thispagestyle{empty}

\newpage

\section{Introduction}

Self-dual (and anti-self-dual) solutions of Yang--Mills equations have played an important role both in mathematics and physics.
Commonly called (anti-)instantons, they are connections with self-dual
curvature on smooth G-bundles over a four dimensional compact manifold $M$. In particular, one considers $\SU(2)$ instantons on the sphere $S^4$.

General solutions are constructed by the ADHM method of \cite{adhm} and it is known \cite{ahs} that the moduli space of $\SU(2)$-instantons, with instanton charge $k$ -- the second Chern class of the bundle -- is a ($8k-3$)-dimensional manifold $\M_k$. For $k=1$ the  moduli space $\M_1$ is isomorphic to an open ball in $\IR^5$ \cite{ha} and, in this construction, generic moduli are obtained from the so-called basic instanton. The latter is though of as a quaternion line bundle over $\IP^1 \IH \simeq S^4$ with connection induced from $\IH^2$ by orthogonal projection. All other instantons of  charge 1 are obtained from the basic one with the action of the conformal group $\SL(2,\IH)$ modulo the isometry group $\Sp(2)=\Spin(5)$. The resulting homogeneous space is $\M_1$; it is also the space of quaternion Hermitean structures in $\IH^2$.

The attempt to generalize to noncommutative geometry the ADHM
construction of $\SU(2)$ instantons together with their moduli
space is the starting motivation  behind papers \cite{gw}, \cite{gw2},
\cite{lpr} and the present one.
A noncommutative principal fibration $\A(S^4_\theta) \into
\A(S^7_{\theta})$ which `quantizes' the classical $\SU(2)$-Hopf fibration over $S^4$, has been constructed in \cite{gw} on the toric noncommutative
four-sphere $S^4_\theta$.
The generators  of $\A(S^4_\theta)$ are the entries of a
projection $p$ which describes the basic instanton on
$\A(S^4_\theta)$.   That is, $p$ gives a projective module of
finite type $p [\A(S^4_\theta)]^4$ and a connection $\nabla = p \circ \dd$ on it, which has a self-dual
curvature and  charge 1, in some appropriate sense; this is the basic instanton. In \cite{gw2} infinitesimal instantons (`the tangent space to the moduli space') were constructed using infinitesimal conformal transformations, that is elements in a quantized enveloping algebra  $\U_\theta(so(5,1))$. In the present paper, we look at a global construction and obtain other charge 1 instantons by
`quantizing' the actions of the Lie groups $\SL(2,\IH)$ and $\Sp(2)$ on the basic instanton which enter the classical construction \cite{atiyah}.

The paper
is organized as follows. In Sect.~\ref{se:principal} we recall the structure of the $\SU(2)$-principal Hopf fibration $S^7_{\theta } \rightarrow S^4_\theta$.
Sect.~\ref{se:SL2H} is devoted to the construction of
$\theta$-deformations $\A(\SL_\theta(2,\IH))$ and
$\A(\Sp_\theta(2))$ of the corresponding classical groups, endowed
with Hopf algebra structures. The algebras $\A(S^7_\theta)$,
$\A(S^4_\theta)$ are then described as quantum homogeneous
spaces of $\A(\Sp_\theta(2))$ as illustrated at the end of the
section. In Sect.~\ref{se:inflated} we consider the coactions of $\A(\SL_\theta(2,\IH))$ and
$\A(\Sp_\theta(2))$ on the Hopf fibration $S^7_{\theta } \rightarrow S^4_\theta$. We use these coactions in Sect. \ref{se:family} to construct a noncommutative family of
instantons by means of the algebra given by the quantum quotient
 of $\A(\SL_\theta(2,\IH))$ by $\A(\Sp_\theta(2))$. This turns out to be a noncommutative algebra generated by 6 elements subject to a `hyperboloid' relation.  We finish by mentioning relations to the notion of quantum families of maps as proposed in \cite{woro,soltan} and by stressing some open problems.

\section{The principal fibration}
\label{se:principal}

The class of deformations that we work with is the one of `toric noncommutative spaces' introduced in \cite{cl} and further elaborated in \cite{cd}. In \cite{gw} a noncommutative principal fibration
$\A(S^4_\theta) \into \mathcal{A}(\S^7)$ was introduced and
infinitesimal instantons on it were constructed in \cite{gw2} using infinitesimal conformal transformations.
We refer to these latter papers for a detailed description of the inclusion $\A(S^4_\theta) \into \mathcal{A}(\S^7)$ as a noncommutative principal fibration (with classical $\SU(2)$ as structure group) and  for its use for noncommutative instantons. Here we limit ourself to a brief description.
The coordinate algebra $\mathcal{A}(\S^7)$ on the sphere $\S^7$ is the $*$-algebra generated by elements
$\{z_j, z_j^* ~; ~j=1,\dots, 4\}$ with relations
\beq \label{7sphere}
z_j z_k = \lambda_{j k} z_k z_j \; , \qquad
z_j^* z_k = \lambda_{k j} z_k z_j^* \; , \qquad
z_j^* z_k^* = \lambda_{j k} z_k^* z_j^* \; ,
\eeq 
and spherical relation $\sum z_j^* z_j  =1$.
The deformation matrix $(\lambda_{j k})$ is taken so to
allow an action by automorphisms of the undeformed group $\SU(2)$  on $\mathcal{A}(\S^7)$ and so that  the subalgebra of invariants under this action is identified with the coordinate algebra $\A(\S^4)$ of a sphere $\S^4$.
With deformation parameter $\lambda=e^{2 \pi \ii \theta}$
the $*$-algebra $\A(\S^4)$ is generated by a central element $x$ and
elements $\alpha,\beta,\alpha^*,\beta^*$ with commutation relations
\beq\label{4sphere}
\alpha \beta = \lambda \beta \alpha \; , \quad
\alpha^* \beta^* = \lambda \beta^* \alpha^*\; , \quad
\beta^* \alpha  = \lambda \alpha  \beta^* \; , \quad
\beta \alpha^* = \lambda \alpha^* \beta \;, \eeq and spherical
relation $\alpha^* \alpha + \beta^*\beta + x^2 =1$.
All this (including the relation between the deformation parameter for $\S^7$ and $\S^4$) is most easily seen by taking the generators of $\A(\S^4)$ as the entries of a projection which yields an `instanton bundle' over $\S^4$. Consider the matrix-valued function on $\S^7$ given by
\beq\label{utheta}
u=(\ket{\psi_1},\ket{\psi_2})=
\begin{pmatrix}
 \uno & -\dueb & \tre & -\qub \\
\due & \unob &  \qu &\treb
\end{pmatrix}^t
,\eeq
where $^t$ denotes matrix transposition, and  $\ket{\psi_1},\ket{\psi_2}$ are elements in the right $\A(S^7_\theta)$-module
$\IC^4 \ot \A(S^7_\theta)$. They are orthonormal with respect to the $\A(S^7_\theta)$-valued Hermitean
structure $\left\langle \xi , \eta \right\rangle= \sum \xi^*_j \eta_j$ and as a consequence,
$u^* u=\II_2$. Hence the
matrix
\beq\label{basicp}
p=uu^* = \ket{\psi_1} \bra{\psi_1}+ \ket{\psi_2} \bra{\psi_2}
\eeq
is a self-adjoint idempotent with entries  in  $\A(\S^4)$; we have explicitly:
\beq\label{ptheta}
p= \frac{1}{2}
\begin{pmatrix}
1 +x & 0 & \alpha & \beta
\\
0 & 1+x & -\mu\,  \beta^* & \bar\mu\, \alpha^*
\\
\alpha^* & -\bar\mu\,  \beta & 1-x &0
\\
\beta^* & \mu\,  \alpha &0 & 1-x
\end{pmatrix}
,\eeq
with $\mu=\sqrt{\lambda}=e^{\pi \ii \theta}$.
The generators of $\A(\S^4)$, bilinear in those of $\A(\S^7)$, are given by
\beq \label{s4ins7}
\alpha=2(\uno \treb + \due \qub) \; , \quad  \beta=2(-\uno \qu + \due \tre) \; ,
\quad
x=\uno\unob + \due \dueb - \tre \treb - \qu \qub \;.
\eeq
The defining relation of the algebra $\A(\S^7)$ can be given on the entries of the matrix $u$ in \eqref{utheta}. Writing $u= (u_{ia})$, with
$i,j=1,\ldots,4$ and $a=1,2$, one gets
\beq\label{commutus}
u_{ia}u_{jb}=\eta_{ji}u_{jb}u_{ia} \; .
\eeq
with $\eta=(\eta_{ij} )$ the matrix
\beq\label{eta}
\eta = \begin{pmatrix} 1 & 1 & \bar\mu & \mu \\ 1 & 1 & \mu & \bar\mu \\
 \mu & \bar{\mu} &1 & 1\\ \bar{\mu} & \mu &1 & 1 \end{pmatrix}.
\eeq
The deformation matrix $(\lambda_{j k})$ in \eqref{7sphere} is  just the above $\eta$ with entries rearranged.

The finitely generated projective $\A(S^4_\theta)$-module $\ce=p[\A(\S^4)]^4$ is isomorphic to the  module
of equivariant maps from $\mathcal{A}(\S^7)$ to $\mathbb{C}^2$ describing the vector bundle associated via the fundamental representation of $\SU(2)$. On
$\E$ one has the Grassmann connection
\beq\label{basic}
\nabla :=p \circ \dd: \E \rightarrow
\E \ot_{\A(S^4_\theta)} \Omega^1(S^4_\theta)
,\eeq
with $\Omega(S^4_\theta)$ a natural differential calculus on $S^4_\theta$. There is also a natural Hodge star operator $*_\theta$ (see below). The connection has a self-dual curvature $\nabla^2=p(dp)^2$, that is,
$$
*_\theta \left( p(\dd p)^2 \right)=p(\dd p)^2
;
$$
 Its `topological charge' is computed to be 1 by a noncommutative index theorem. This `basic' noncommutative instanton has been given a twistor description in \cite{brain}.

\bigskip

The two algebras $\A(\S^7)$ and $\A(S_\theta^4)$ can be
described in terms of a deformed (a `star') product on the undeformed algebras
$\A(S^7)$ and $\A(S^4)$. Both spheres $S^7$ and $S^4$ carry an action of the torus $\IT^2$, which is compatible with the action of $\SU(2)$ on the total space $S^7$. In other
words, it is an action on the principal $\SU(2)$-bundle $S^7 \to S^4$.
The action by automorphisms on the algebra
$\A(S^4)$ is given simply by
$$
\sigma_t: (x,\alpha,\beta) ~\mapsto~ (x, e^{2\pi \ii t_1}\alpha, e^{2\pi \ii t_2}\beta) ,
$$
for $t = (t_1,t_2) \in \IT^2$. Now, any polynomial in the algebra $\A(S^4)$ is  decomposed into elements which are homogeneous under this action.
An element $f_r  \in \A(S^4)$ is said to be homogeneous of bidegree $r=(r_1,r_2)\in \IZ^2 $ if
$$
\sigma_t(f_r)=e^{2\pi \ii (r_1 t_1 + r_2 t_2)} \,f_r ,
$$
and each $f \in \A(S^4)$ is a unique finite sum of homogeneous elements \cite{Ri93}. This decomposition corresponds to writing the polynomial $f$ in terms of monomials in the generators.

Let now $\theta = (\theta_{j k} = - \theta_{k j})$ be a real antisymmetric $2\times 2$ matrix (thus given by a single real number, $\theta_{12}=\theta$, say).
The $\theta$-deformation of $\A(S^4)$ is defined by replacing the ordinary product by a deformed product, given on homogeneous elements by
\beq
\label{eq:star-product}
f_r \times_\theta g_{s} := e^{ \pi \ii \theta (r_1s_2 - r_2 s_1)}f_r g_{s} ,
\eeq
and extended linearly to all elements in $\A(S^4)$. The vector space $\A(S^4)$ equipped with the product $\times_\theta$ is denoted by $\A(S^4_\theta)$. On the other hand, the algebra $\A(S^7)$ does not carry an action of this $\IT^2$ but   rather a lifted action of a double cover $2$-torus \cite{gw2}. Nonetheless, the lifted action still allows us to define the algebra $\A(\S^7)$ by endowing the vector space  $\A(S^7)$ with a deformed product similar to the one in \eqref{eq:star-product}. As the notation suggests, these deformed algebras are shown to be isomorphic to the algebras defined by the relations in equations \eqref{7sphere} and \eqref{4sphere}.

In fact, the torus action can be extended to forms and one also deforms the exterior algebra of forms via a product like the one in \eqref{eq:star-product} on spectral components so producing deformed exterior algebras $\Omega(S^4_\theta)$ and $\Omega(S^7_\theta)$. As for functions, these are isomorphic as vector spaces to their undeformed counterparts but endowed with a deformed product.

As mentioned, the spheres $\S^4$ and $\S^7$ are examples of toric noncommutative manifolds (originally called isospectral deformations \cite{cl}).
They have noncommutative geometries via spectral triples whose Dirac operator and Hilbert space of spinors are the classical ones: only the algebra of functions and its action on the spinors is changed. In particular, having an undeformed Dirac operator (or, in other words, an undeformed metric structure) one takes as a Hodge operator $*_\theta$ the undeformed operator on each spectral component of the algebra of forms.

\section{Deformations of the groups $\SL(2,\mathbb{H})$ and $\Sp(2)$}
\label{se:SL2H}

Our interest in deforming the groups $\SL(2,\mathbb{H})$ and $\Sp(2)$ is motivated by their use for the construction of instantons on $S^4$.
Classically, charge 1 instantons are generated from
the basic one by the action of the conformal group $\SL(2,\IH)$ of
$S^4$. Elements of the subgroup $\Sp(2)\subset \SL(2, \IH)$ leave
invariant the basic one, hence  to get
new instantons one needs to quotient $\SL(2, \IH)$ by the spin group $\Sp(2)\simeq \Spin(5)$. The resulting moduli space
of $\SU(2)$ instantons on $S^4$ modulo gauge transformations  is identified  (cf. \cite{atiyah}) with the five-dimensional quotient manifold
$\SL(2,\IH)/\Sp(2)$.

In a parallel attempt to generate instantons on $\A(\S^4)$,
we  construct a quantum group $\SL_\theta(2,\IH)$ and its quantum
subgroup $\Sp_\theta(2)$.  An infinitesimal construction was proposed in \cite{gw2} where a deformed dual enveloping algebra $\U_\theta(so(5,1))$ was used to generate infinitesimal instantons by acting on the basic instanton described above.

The  construction of Hopf algebras $\A(\SL_\theta(2,\IH))$ and $\A(\Sp_\theta(2))$ is a special case of the quantization of compact Lie groups using Rieffel's strategy in \cite{Ri93a}, and studied
for the toric noncommutative geometries in \cite{varilly}. Firstly, a deformed (Moyal-type) product $\times_\theta$ is constructed on the algebra of (continuous) functions $\A(G)$ on a compact Lie group $G$, starting with an action of a closed connected abelian subgroup of $G$ (usually a torus). The algebra $\A(G)$ equipped with the deformed product is denoted by $\A(G_\theta)$. Keeping the classical expression of the coproduct, counit and antipode on $\A(G)$, but now on the algebra $\A(G_\theta)$, the latter becomes a Hopf algebra. It is in duality with a deformation of the universal enveloping algebra $\U(\cg)$ of the Lie algebra $\cg$ of $G$. The Hopf algebra $\U(\cg)$ is deformed to $\U_\theta(\cg)$ by leaving unchanged the algebra structure while twisting the coproduct, counit and antipode. The deformation from $\U(\cg)$ to $\U_\theta(\cg)$ is implemented with a twist of Drinfel'd type \cite{Dr83,Dr90} -- in fact, explicitly constructed in \cite{Re90} for the cases at hand -- as revived in \cite{sitarz}.

The deformed enveloping algebra $\U_\theta(so(5,1))$ was explicitly constructed in \cite{gw2}. We now briefly discuss the dual construction for the Lie group $\SL(2,\IH)$. The torus $\IT^2$ is embedded in $\SL(2,\IH)$ diagonally, 
$$
\rho(t) =\diag
\left(e^{2\pi \ii t_1}, \, e^{2\pi \ii t_2}\right),
\qquad \mathrm{for}\quad t=(t_1,t_2)\in\IT^2 ,
$$
and the group $\IT^2 \times \IT^2$ acts on $\SL(2,\IH)$ by
\beq\label{act;co}
(s,t,g) \in \IT^2 \times \IT^2 \times \SL(2,\IH) \mapsto \rho(s) \cdot g \cdot \rho(t)^{-1} \in \SL(2,\IH).
\eeq
Similar to the case of the spheres in Sect.~\ref{se:principal}, any function $f \in \A(\SL(2,\IH))$ expands in a series $f= \sum_r f_r$ of homogeneous elements for this action of $\IT^4$, but now $r=(r_1,r_2,r_3,r_4)$ is a multidegree taking values in $\IZ^4$. A deformed product $\times_\theta$ is defined by an analogue of formula \eqref{eq:star-product} on homogeneous elements:
$$
f_r \times_\theta g_s = e^{\pi i \theta (r_1 s_2 - r_2 s_1 + r_3 s_4 - r_4 s_3)} f_r g_s,
$$
and extended by linearity to the whole of $\A(\SL(2,\IH))$.
The resulting deformed algebra $\A(\SL_\theta(2,\IH))$, endowed with the classical (expressions for the) coproduct  $\Delta$, counit $\epsilon$ and antipode $S$ becomes a Hopf algebra.

In fact, to avoid problems coming from the noncommutativity of quaternions, we shall think of elements in $\IH$ as 2 by 2 matrices over $\IC$ via the natural inclusion
$$
\IH \ni q = c_1 + c_2 j \quad\mapsto\quad \begin{pmatrix} c_1 & c_2 \\ -\bar{c}_2 & \bar{c}_1 \end{pmatrix} \in \Mat_2(\IC), \qquad \mathrm{for} \quad c_1,c_2 \in \IC.
$$

In the present paper, we need not only the Hopf algebra $\A(\SL_\theta(2,\IH))$ but also its coaction on the principal bundle $\A(S^4_\theta) \into \mathcal{A}(\S^7)$, and in turn on the basic instanton connection \eqref{basic} on the bundle, in order to generate new instantons. Having this fact in mind we proceed to give an explicit construction
of $\A(\SL_\theta(2,\IH))$ out of its coaction in a way that also shows its
{\it quaternionic} nature.

\subsection{The quantum group $\SL_\theta(2,\IH)$}
\label{sse:SL2H}

For the deformation of the {\it quaternionic} special linear group $\SL(2,\IH)$, we start from the algebra of a two-dimensional deformed quaternionic space $\IH^2_\theta$. Let $\A(\IC^4_\theta)$ be the $*$-algebra generated by elements $\{z_j, z_j^* ~;~ a=1,\ldots,4 \}$ with the relations as in equation \eqref{7sphere} (for the specific value of the deformation parameter $\lambda$ considered in Sect.~\ref{se:principal} and obtained from \eqref{eta} as mentioned there) but without the spherical relation that defines $\A(S^7_\theta)$. We take $\A(\IH_\theta^2)$ to be the algebra $\A(\IC^4_\theta)$ equipped with the antilinear $*$-algebra map $j:\A(\IC^4_\theta) \to \A(\IC^4_\theta)$ defined on generators by
\begin{align*}
\label{eq:j}
j: (z_1, z_2, z_3, z_4) \mapsto (\due, -\uno, \qu, -\tre).
\end{align*}
It is worth stressing that this deformation of the quaternions takes place between the two copies of $\IH$ while leaving the quaternionic structure within each copy of $\IH$ undeformed. Since the second column of the matrix $u$ in \eqref{utheta} is the image through $j$ of the first one, we may think of $u$ as made of two deformed quaternions.

Following a general strategy \cite{wang}, we now define $\A(M_\theta(2,\IH))$ to be the universal bialgebra for which $\A(\IH_\theta^2)$ is a comodule $*$-algebra. More precisely, we define a {\it transformation bialgebra} of $\A(\IH_\theta^2)$ to be a bialgebra $\B$ such that there is a $*$-algebra map
\begin{align*}
\Delta_L:\A(\IC^4_\theta) \rightarrow \B \ot \A(\IC^4_\theta),
\end{align*}
which satisfies
\beq
\label{condition:j} (\id \ot j) \circ\Delta_L=\Delta_L \circ j .
\eeq
We then set $\A(M_\theta(2,\IH))$ to be the {\it universal transformation bialgebra} in the following sense: for any transformation bialgebra $\B$ there exists a morphism of transformation bialgebras ({i.e.} commuting with the coactions) from $\A(M_\theta(2,\IH))$ onto $\B$.

The requirement that $\A(\IH_\theta^2)$ be a $\A(M_\theta(2,\IH))$-comodule algebra allows us to derive the commutation relations of the latter.
A coaction $\Delta_L$ is given by matrix multiplication,
\beq\label{aactc4}
\Delta_L:  (z_1,- \dueb, z_3,- \qub)^t \mapsto A_\theta \pot
(z_1, - \dueb, z_3, -\qub)^t ,
\eeq
for a generic $4\times4$ matrix $A_\theta=(A_{ij})$.
{}Asking for  \eqref{condition:j} we have
$$
(A_{jk})^* = (-1)^{j+k } A_{j'  k'},
$$
with $j ' = j +(-1)^{j+1}$ and the same for $k'$; this means that $A_\theta$ has  the form
\beq\label{Atheta} A_\theta
=
\begin{pmatrix}
a_{ij} & b_{ij}
\\
c_{ij} & d_{ij}
\end{pmatrix}
=
\begin{pmatrix}
a_1 & a_2 & b_1 & b_2
\\
-\ba_2 & \ba_1 & -\bb_2 &  \bb_1
\\
c_1 & c_2 & d_1 & d_2
\\
- \bc_2 &  \bc_1 & - \bd_2 & \bd_1
\end{pmatrix} \;.
\eeq
We use `quaternion notations' for the above matrix and write
\beq\label{Atheta2}
A_\theta
= \begin{pmatrix} a & b \\ c & d \end{pmatrix} , \qquad \mathrm{with} \quad a=(a_{ij})=\begin{pmatrix} a_1 & a_2 \\ -\ba_2 & \ba_1  \end{pmatrix},
\eeq
and similarly for the remaining parts.

The defining matrix in \eqref{Atheta} has a `classical form'. One readily finds that with respect to the torus action \eqref{act;co} its entries $A_{ij}$ are of multidegree $\Lambda_i \oplus (- \Lambda_j)$ in $\IZ^4$ with
$$
\Lambda = (\Lambda_i) = \big( (1,0), (-1,0), (0,1), (0,-1) \big).
$$
The general strategy exemplified by the deformed product
\eqref{eq:star-product} would then give the deformed product and  in turn, the commutation relations defining the deformed algebra $\A(M_\theta(2,\mathbb{H}))$. We shall get them directly from the coaction on the algebra $\A(\IC^4_\theta)$.

The transformations induced on the generators of
$\A(\IC^4_\theta)$ reads
\beq
\label{transf}
\begin{aligned}
\wuno := \Delta_L(\uno)= a_1 \otimes \uno - a_2 \otimes \dueb + b_1 \otimes \tre - b_2 \otimes \qub
\\
\wdue  := \Delta_L(\due)=a_1 \otimes \due +  a_2\otimes \unob + b_1\otimes \qu + b_2\otimes \treb
\\
\wtre :=  \Delta_L(\tre)=c_1 \otimes\uno - c_2 \otimes\dueb + d_1\otimes \tre - d_2\otimes \qub
\\
\wqu  := \Delta_L(\qu)=c_1 \otimes\due + c_2 \otimes\unob +d_1 \otimes\qu +d_2 \otimes\treb
\end{aligned}
\eeq
with $\Delta_L (z_j^*) =  (\Delta_L(z_j))^*$. The
condition for $\Delta_L$ to be an algebra map determines  the
commutation relations among the generators of
$\A(M_\theta(2,\mathbb{H}))$:
the algebra generated by the
$a_{ij}$ is commutative, as well as the algebras generated by the
$b_{ij}$, $c_{ij}$ and the $d_{ij}$. However, the whole algebra is
not commutative and there are relations. A straightforward computation allows one to concisely write them as
\beq\label{algAtheta}
A_{ij} A_{kl} = \eta_{ki} \eta_{jl} A_{kl} A_{ij}
\eeq
with $\eta=(\eta_{ki})$ the deformation matrix in \eqref{eta}.
Indeed, imposing that \eqref{aactc4} defines a $*$-algebra map on the generators of $\A(\IC^4_\theta)$, and using the relations \eqref{commutus}, we have $\sum_{kl} (A_{ik}A_{jl}- \eta_{ji}\eta_{kl}A_{jl}A_{ik}) \ot u_{ka}u_{lb}=0$.
Since for $a\leq b$ the elements $u_{ka}u_{lb}$ could be taken to be all independent, relations $A_{ik}A_{jl}- \eta_{ji}\eta_{kl}A_{jl}A_{ik}=0$ hold, for all values of $a,b$.

An explicit expression of the above commutation relations is in
App.~\ref{Ap-acr}.
It is not difficult to see that $\A(M_\theta(2,\IH))$ is indeed the universal transformation bialgebra, since the commutation relations \eqref{algAtheta} and the quaternionic structure of $A_\theta$ in \eqref{Atheta} are derived from the minimal requirement of $\Delta_L$ to be a $*$-algebra map such that \eqref{condition:j} holds.

\bigskip
In order to define the quantum group $\SL_\theta(2,\IH)$ we need a determinant. This is most naturally introduced via the coaction on forms.
There is a natural differential calculus
$\Omega(\IC^4_\theta)$ generated in degree 1 by elements
$\{\dd z_j, a=1,\dots,4\}$ and relations
\begin{align*}
& z_j \dd z_k - \lambda_{j k} \dd z_k z_j = 0, \qquad
z_j \dd z_k^* - \lambda_{k j} \dd z_k^* z_j = 0 , \qquad
z_j^* \dd z_k - \lambda_{k j} \dd z_k z_j^* = 0 , \\
& \dd z_j \dd z_k + \lambda_{j k} \dd z_k \dd z_j = 0, \qquad
\dd z_j \dd z_k^* + \lambda_{k j} \dd z_k^* \dd z_j = 0 .
\end{align*}
together with their conjugates.
The forms $\Omega(\IC^4_\theta)$ could be obtained from the general procedure mentioned at the end of Sect.~\ref{se:principal}.  The result is also isomorphic to the one obtained from the general construction \cite{connes} which uses the Dirac operator to implement the exterior derivative as a commutator.

The coaction $\Delta_L$ is extended to forms by requiring it to commute with
$\dd$.
Having the action \eqref{aactc4}, it is natural to
define a determinant element by setting
$$
\Delta_L(\dd z_1 \dd z_2^* \dd z_3 \dd z_4^*) =:
\det(A_\theta) \otimes \dd z_1 \dd  z_2^* \dd z_3 \dd  z_4^* .
$$
We find its explicit form by using the relations of $\Omega(\IC^4_\theta)$:
\begin{align} \label{det}
\det(A_\theta) &= a_1[\ba_1(d_1 \bd_1 + d_2 \bd_2)+\bb_2(\mu c_2 \bd_1 -d_2 \bc_1)-\bb_1(\bar\mu c_2 \bd_2 + d_1 \bc_1)] \nn
\\
&\quad - a_2[-\ba_2(d_1 \bd_1 +d_2\bd_2 ) +\bb_2(\bar\mu c_1 \bd_1 +d_2 \bc_2)+\bb_1 (-\mu c_1 \bd_2 +d_1 \bc_2)] \nn
\\
&\quad + b_1[-\ba_2(c_2\bd_1 - \bar\mu d_2 \bc_1)-\ba_1(c_1 \bd_1 + \mu d_2 \bc_2)+\bb_1(c_1 \bc_1 +c_2 \bc_2)]\nn
\\
&\quad - b_2[\ba_2(c_2 \bd_2 + \mu d_1 \bc_1)+\ba_1(c_1 \bd_2 - \bar\mu d_1 \bc_2)-\bb_2(c_1 \bc_1 + c_2 \bc_2)].
\end{align}
A more compact form for $\det (A_\theta)$ is found to be (see also App.~\ref{Ap-B})
\beq\label{det-cond}
\det (A_\theta)=\sum_{\sigma \in S_4} (-1)^{|\sigma|} \varepsilon^{\sigma}
A_{1,\sigma(1)}A_{2,\sigma(2)}A_{3,\sigma(3)}A_{4,\sigma(4)}, 
\eeq
with $\varepsilon^{\sigma}=\varepsilon^{\sigma(1)\sigma(2)\sigma(3)\sigma(4)}$. The tensor $\varepsilon^{ijkl}$ has components 
$$
\varepsilon^{1324}= \varepsilon^{cycl}=  \mu
 \quad ; \quad
\varepsilon^{1423}= \varepsilon^{cycl}= \bar \mu ,
$$
and equal to $1$ otherwise.
In the limit $\theta \rightarrow 0$, the element $\det(A_\theta)$ reduces to the determinant of the matrix $A_{\theta=0}$ as it should.
Additional results on the determinant are in the following lemmata.
\begin{lemma}\label{l1}
For each $i,l \in \{1,\dots,4 \}$,
define the corresponding algebraic complement:
$$
\hat A_{il}=
\sum_{\sigma \in S_3} (-1)^{|\sigma|}
\varepsilon^{\sigma_1\dots \sigma_{i-1} l \sigma_{i+1} \dots \sigma_4}
\eta^{\sigma_1 l} \eta^{\sigma_2 l} \cdots \eta^{\sigma_{i-1}l}
A_{1,\sigma_1}\dots A_{i-1,\sigma_{i-1}} A_{i+1,\sigma_{i+1}} \dots A_{4,\sigma(4)},
$$
where $\sigma=(\sigma_1, \dots \sigma_{i-1},\sigma_{i+1},\dots \sigma_4)=\sigma(1, \dots l-1,l+1,\dots 4)\in S_3$, the group of permutations of three objects.
Then,
$$A_{il} \hat A_{il}= \hat A_{il} A_{il} \; \qquad \mathrm{for ~any}\; i\;
\mathrm{and} \; l .
$$
\end{lemma}
\begin{proof}
We use the shorthand $c^\sigma_l=
\varepsilon^{\sigma_1\dots \sigma_{i-1} l \sigma_{i+1} \dots \sigma_4}
\eta^{\sigma_1 l} \eta^{\sigma_2 l} \cdots \eta^{\sigma_{i-1}l} $.
Then, the commutation relations \eqref{algAtheta} yield,
\begin{align*}
A_{il} \hat A_{il} &= \sum_{\sigma \in S_3(\hat l)} c^\sigma_l A_{il}A_{1\sigma_1}\dots A_{i-1,\sigma_{i-1}} A_{i+1,\sigma_{i+1}} \dots A_{4,\sigma(4)} =\\
&=
\sum_{\sigma \in S_3(\hat l)} c^\sigma_l
\big( \eta^{1 i}\cdots \eta^{i-1,i} \eta^{i+1,i} \cdots \eta^{4i} \big)
\big( \eta^{l, \sigma_1 }\cdots \eta^{l \sigma_{i-1}} \eta^{l, \sigma_{i+1}} \cdots \eta^{l4} \big)
A_{1\sigma_1}\dots A_{i-1,\sigma_{i-1}} \cdot
\\
& \qquad \qquad  \qquad \qquad \qquad  \qquad \cdot A_{i+1,\sigma_{i+1}} \dots A_{4,\sigma(4)}A_{il} \\
&=  \hat A_{il} A_{il} \; .
\end{align*}
Here we used the fact that $\eta^{1 i}\cdots \eta^{i-1,i} \eta^{i+1,i} \cdots \eta^{4i}$ is the product of the elements in the $i$-th column of $\eta$ excluded the element $\eta^{ii}=1$ and the result is $1$ as one deduces from the form of the matrix $\eta$ in \eqref{eta}. Similarly for the other coefficient given by the product of the elements of the $l$-th row.
\end{proof}
\begin{lemma}\label{l2}
The determinant $\det(A_\theta)$ is computed via a Laplace expansion:
\begin{enumerate}
\item by rows; for each $i \in \{1,\dots,4 \}$ fixed:
$
\qquad \det (A_\theta)=\sum_{l} (-1)^{i+l} A_{il} \hat A_{il} ;
$
\item by columns; for each $i \in \{1,\dots,4 \}$ fixed:
$
\qquad \det (A_\theta)=\sum_{l} (-1)^{i+l} A_{li} \hat A_{li} .
$
\end{enumerate}
\end{lemma}
\begin{proof}
These follow from \eqref{algAtheta} after a lengthy but straightforward computation.
\end{proof}

The particular form of the deformation matrix $\eta_{ij}$ defining the relations in $\A(\S^7)$ implies that $\det(A_\theta)$ is (not surprisingly) a central element in the algebra $\A(M_\theta(2,\mathbb{H}))$ generated by the entries of $A_\theta$. Hence we can take the quotient of this algebra by the two-sided ideal generated by $\det(A_\theta) -1$; we will denote this quotient by $\A(\SL_\theta(2,\IH))$. The image of the elements $A_{ij}$ in the quotient algebra will again be denoted by $A_{ij}$.

\bigskip

In order to have a quantum group we need more structure. On the algebra $\A(\SL_\theta(2,\IH))$ we define a coproduct by
$$
\Delta(A_{ij}):=\sum\nolimits_k A_{ik} \otimes A_{kj} ,
$$
a counit by
$$
\epsilon(A_{ij}):=\delta_{ij},
$$
whereas the antipode $S$ is defined by
$$
S(A_{ij}):= (-1)^{i+j} \hat A_{ji} \, .
$$
Here $\hat A_{li}$ are the algebraic complements introduced in Lemma~\ref{l1}.
Indeed, from Lemma~\ref{l2},
$ \sum_l A_{il} S(A_{li})= \sum_l (-1)^{i+l} A_{il} \hat A_{il} =\det(A_\theta)=1$,
and similarly, using also Lemma \ref{l1},
$ \sum_l  S(A_{il}) A_{li}= \sum_l (-1)^{i+l} \hat A_{li}  A_{li} =\det(A_\theta) =1$.
Moreover, if $i \neq j$, $\sum_l A_{il} S(A_{lj})= 0$ as one shows by explicit computation.

The definitions above are collected in the following proposition.
\begin{prop}\label{prop:sl}
The datum $(\A(\SL_\theta(2,\IH)),\Delta,\epsilon,S)$ constitutes  a Hopf algebra.
\end{prop}

\bigskip
The coaction $\Delta_L$ in \eqref{aactc4} passes to a coaction of $\A(\SL_\theta(2,\IH))$ on $\A(\IH_\theta^2)$ and it is still a $*$-algebra map.
However, the spherical relation $\sum_j z_j^* z_j = 1$ is no longer invariant under $\Delta_L$. Thus, the algebra $\A(S^7_\theta)$ is not an $\A(\SL_\theta(2,\IH))$-comodule algebra but only a $\A(\SL_\theta(2,\IH))$-comodule.   We shall elaborate more on this in Sect.~\ref{se:inflated} below.

\subsection{The quantum group $\Sp_\theta(2)$}\label{se:sp}

Motivated by the classical picture we next introduce the symplectic group $\A(\Sp_\theta(2))$. Recall that a two-sided $*$-ideal $I$ in a Hopf algebra $(\A,\Delta,\epsilon,S)$ is a Hopf ideal if
\beq\label{condhi}
\Delta(I) \subseteq I \otimes \A + \A \otimes I, \qquad \epsilon(I)=0, \qquad S(I) \subseteq I.
\eeq
Then the quotient $\A / I$ is a Hopf algebra with induced structures
$((\pi \ot \pi)\circ \Delta, \epsilon, \pi \circ S)$, where $\pi:\A \to \A/ I$ is the natural projection.
\begin{prop} Let $I$ denote the two-sided $*$-ideal in $\A(\SL_\theta(2,\IH))$
generated by the elements $\sum_k (A_{ki})^* A_{kj}-\delta_{ij}$ for $i,j=1,\ldots,4$. Then $I \subset \A(\SL_\theta(2,\IH))$ is a Hopf ideal.
\end{prop}
\begin{proof}  The first two conditions in \eqref{condhi} follow easily from the definition of $\Delta$ and $\epsilon$ for $\A(\SL_\theta(2,\IH))$ in Proposition~\ref{prop:sl}. For the third, we observe that if $J$ is an ideal in $\A(\SL_\theta(2,\IH))$ such that the classical counterpart  $J^\class$ is an ideal in $\A(\SL(2,\IH))$ which is generated by homogeneous elements, then $J = J^\class$ as vector spaces. Indeed, the deformed product of a generator with any homogeneous function merely results in multiplication by a complex phase.
In our case, the classical counterparts for the generators
$\sum_k (A_{ki})^* A_{kj} - \delta_{ij}$ are indeed homogeneous (if $i=j$, they are of bidegree $(0,0)$, otherwise of bidegree $\Lambda_i - \Lambda_j$) and the above applies. In particular, $S(I) = S(I^\class) \subseteq I^\class = I$.
\end{proof}

\begin{cor}
The quotient
$\A(\Sp_\theta(2)):=\A(\SL_\theta(2,\IH))/I$ is a Hopf algebra
with the induced Hopf algebra structure.
\end{cor}
\noindent
We still use the symbols $(\Delta,\epsilon,S)$ for the induced structures. The `defining matrix' $A_\theta$ of $\A(\Sp_\theta(2))$ has the form \eqref{Atheta} (or \eqref{Atheta2}) with the additional condition that $A^*_\theta A_\theta=1$, coming from the very definition of $\A(\Sp_\theta(2))$. A little algebra shows also that  $A_\theta A^*_\theta=1$. These conditions are equivalent to the statement that
$S(A_\theta)= A_\theta^*$. In the quaternionic form, the conditions $A^*_\theta A_\theta = A_\theta A^*_\theta =1$ becomes 
$$
\begin{pmatrix}
a^* a + c^* c& a^*b+ c^*d   \\
b^* a + d^*c & b^*b + d^*d
\end{pmatrix}
=
\begin{pmatrix}
a a^*  + b b^* & ac^*+ bd^*   \\
ca^*  + db^* & cc^* + dd^*
\end{pmatrix}  =
\begin{pmatrix}
1 & 0 \\ 0 & 1
\end{pmatrix} .
$$

\subsection{The quantum homogeneous spaces $\A(S^7_\theta)$ and $\A(S^4_\theta)$}
\label{s7qhs}

Using the same notations as in \eqref{Atheta},
let us consider the two sided ideal in $\A(\Sp_\theta (2))$ given by
$
I_\theta := \langle b_{ij}, c_{ij}, a_2, \ba_2, a_1 -1,  \ba_1 -1\rangle $.
This is a $*$-Hopf ideal, that is,
$$
\Delta(I_\theta) \subseteq \A(\Sp_\theta (2)) \ot I_\theta+
 I_\theta \ot  \A(\Sp_\theta (2)) \; , \qquad
\varepsilon (I_\theta)=0 \; , \qquad
 S(I_\theta) \subseteq I_\theta \; ,
 $$
 and we can take the quotient Hopf algebra $\A(\Sp_\theta(1)):=
\A(\Sp_\theta (2)) / I_\theta$ with corresponding
projection map $\pi_{I_\theta}$. By projecting with $\pi_{I_\theta}$, the algebra
reduces to the commutative one generated by the entries of the
diagonal matrix $\pi_{I_\theta}(A_\theta)
=\diag(\II_2, d_{ij})=\diag(\II_2, d)$, with
$d^* d=d d^* = \II_2$  or $d_1 \bd_1 + d_2 \bd_2 =1$;
hence, $\A(\Sp_\theta(1))= \A(\Sp(1))$.
There is a coaction
 \begin{eqnarray*}
\A(\Sp_\theta(2)) \rightarrow \A(\Sp_\theta(2)) \ot
\A(\Sp(1)),\qquad
\begin{pmatrix} a  & b \\
c  & d
\end{pmatrix}  \mapsto  \begin{pmatrix} a  & b
\\
c  & d
\end{pmatrix} \pot \begin{pmatrix} \II & 0
\\
0 & d
\end{pmatrix} \;.
\end{eqnarray*}
The corresponding algebra of coinvariants
$\A(\Sp_\theta(2))^{co(\A(\Sp(1))}$  is clearly generated by the first two columns
$(a,c)=\{a_{ij}, c_{ij}\}$ of $A_\theta$. An algebra isomorphism between the algebra of coinvariants and
$\A(S^7_\theta)$ is provided by the $*$-map sending these columns to the matrix $u$ in \eqref{utheta}. On the generators this is given by
\beq\label{corris}
a_1 \mapsto \uno \;, \quad a_2 \mapsto \due \;, \quad c_1 \mapsto \tre \;, \quad c_2 \mapsto \qu
\eeq
and the spherical relation corresponds to the condition
$(A_\theta^* A_\theta)_{11}= \sum (\ba_i a_i  + \bc_i c_i)= 1$.
Summarizing, we have that $\A(\Sp_\theta(2))^{co(\A(\Sp(1))} \simeq
\A(S^7_\theta)$. It follows from the general theory of noncommutative principal bundles over quantum homogeneous spaces \cite{BM93} that the inclusion $\A(S^7_\theta)\into\A(\Sp_\theta(2))$ is a noncommutative principal bundle with the classical group $\Sp(1)$ as structure group.

\bigskip

Next, we consider the  ideal in $\A(\Sp_\theta(2))$ given
by $J_\theta := \langle b_{ij}, c_{ij} \rangle$, --
again easily shown to be a Hopf ideal. Denote by
$\pi_{J_\theta}$ the projection map onto the quotient Hopf algebra
$\A(\Sp_\theta (2)) / J_\theta$ generated, as an algebra,
by the entries of
$\pi_{J_\theta}(A_\theta) = \diag(a_{ij}, d_{ij})= \diag(a, d)$. The conditions
$A_\theta^* A_\theta=A_\theta A_\theta^*=1$ give that both $\{a_{ij}\}$ and $\{d_{ij}\}$ generate a copy of the algebra $\A(\Sp(1))$. However, from the explicit relations in App.~\ref{Ap-acr} we see that in general
$a_{ij} d_{mn} \neq d_{mn} a_{ij}$ and the quotient algebra is not commutative.

The algebra of coinvariants  for the right coaction $(\id \ot \pi_{J_\theta})\circ \Delta$ of $\pi_{J_\theta}(\A(\Sp_\theta(2)))$ on
$\A(\Sp_\theta(2))$, is $\A(S^4_\theta)$.
Indeed, with the map $(\id
\ot \pi_{J_\theta})\circ \Delta  :  A_\theta \mapsto A_\theta \pot
\pi_{J_\theta}(A_\theta)$ given by
$$
\begin{pmatrix} a & b
\\
c & d
\end{pmatrix} \mapsto  \begin{pmatrix} a \pot a & b \pot d
\\
c\pot a & d \pot d
\end{pmatrix} \; ,
$$
one finds that the $*$-algebra of coinvariants is generated by the elements
\begin{gather*}
 (aa^*)_{11}= a_1 \ba_1+ a_2 \ba_2   \; ,    \qquad
 (ca^*)_{11}= c_1 \ba_1+ c_2 \ba_2   \; , \qquad
 (ca^*)_{12}= - c_1  a_2+ c_2  a_1    \;.
\end{gather*}
The $*$-map \eqref{corris} combined with \eqref{s4ins7} then provides the
identification with the generators of $\A(S^4_\theta)$. Again, the general theory of noncommutative principal bundles over quantum homogeneous spaces of \cite{BM93} gives that the inclusion $\A(S^4_\theta)\into\A(\Sp_\theta(2))$ is a noncommutative principal bundle with $\pi_{J_\theta}(\A(\Sp_\theta(2))$ as structure group.  It is a deformation of the principal bundle over $S^4$ with total space $\Sp(2)$ and structure group $\Sp(1)\times \Sp(1)$.

\section{Noncommutative conformal transformations}\label{se:inflated}

There is a natural coaction $\Delta_L$ of $\A(\SL_\theta(2,\IH))$ on
the $\SU(2)$ noncommutative principal fibration $\A(\S^4) \into \A(\S^7)$ of
Sect.~\ref{se:principal}. Since the matrix $u$ in \eqref{utheta}
consists of two deformed quaternions, the left coaction $\Delta_L$ of $\A(\SL_\theta(2,\IH))$ in \eqref{transf}
can be written on $\A(S^7_\theta)$ as
\beq\label{transu}
\Delta_L : \A(S^7_\theta) \to \A(\SL_\theta(2,\IH)) \ot \A(S^7_\theta),
\quad u\mapsto \widetilde u := \Delta_L (u)= A_\theta \pot u  ,
\eeq
or, in  components,
\beq \label{left-coaction:theta}
u_{ia} \mapsto \widetilde u_{ia} := \Delta_L(u_{ia}) = \sum\nolimits_j A_{ij} \otimes u_{ja}.
\eeq
We have already mentioned that the left coaction $\Delta_L$ of $\A(\SL_\theta(2,\IH))$ as in \eqref{transf} does not leave invariant the spherical relation:  $\Delta_L(\sum_j z_j^* z_j)
\not= 1\ot 1$.
We will denote by $\A(\widetilde S^7_\theta)$ the image of $\A(\S^7)$ under the left
coaction of $\A(\SL_\theta(2,\IH))$: it is a subalgebra of
$\A(\SL_\theta(2,\IH)) \otimes \A(\S^7)$.
We think of $\A(\widetilde S^7_\theta)$ as a $\theta$-deformation of
a family of `inflated' spheres. Since $\sum_j z_j^* z_j$ is central in $\A(\S^7)$ its image
\beq \label{def:rho} \rho^2:=\Delta_L(\sum\nolimits_j  z_j^*  z_j),
\eeq
is a central element in $\A(\widetilde S^7_\theta)$ that  parametrizes a
family of noncommutative 7-spheres $\widetilde S^7_\theta$.
By evaluating $\rho^2$ as any real number $r^2 \in \IR$,
we obtain an algebra $\A(S^7_{\theta,r})$ which is a deformation
of the algebra of polynomials on a sphere of radius $r$.

\begin{rem}\label{rem:rho}
As expected, the coaction of the quantum subgroup $\A(\Sp_\theta(2))$
does not `inflate the spheres', {i.e.} $\rho^2=1\ot 1$ in this case.
Indeed, if $A_{\theta}=(A_{ij})$ is the defining
matrix of $\A(\Sp_\theta (2))$, one gets
$$
(u^*u)_{ab} \mapsto \sum\nolimits_{ijl} (A^*)_{li} A_{ij} \ot (u^*)_{al} u_{jb} =
\sum\nolimits_{jl} \delta_{lj}\ot (u^*)_{al}u_{jb} = 1 \ot (u^*u)_{ab} \; ,
$$
which gives $\sum_j  z_j^*  z_j \mapsto 1\ot \sum_j  z_j^*  z_j$.
Hence, both $\A(\S^7)$ and $\A(\S^4)$ are $\A(\Sp_\theta(2))$-comodule $*$-algebras.
Using the identification \eqref{corris} one sees that the coaction of
$\A(\Sp_\theta(2))$ on $\A(\S^7)$ is the restriction of the coproduct of $\A(\Sp_\theta(2))$  to the first column of $A_\theta$, i.e. to the algebra of coinvariants $\A(\Sp_\theta(2))^{co(\A(\Sp(1))}$.
\end{rem}

Next, we define a right action of $\SU(2)$ on $\A(\widetilde{S}^7_\theta)$ in such a way that the corresponding algebra of invariants describes a family of noncommutative 4-spheres. It is natural to require that the above left coaction of $\A(\SL_\theta(2,\IH))$ on $\A(\S^7)$ intertwines the right action
of $\SU(2)$ on $\A(\widetilde{S}^7_\theta)$ with the action of $\SU(2)$ on $\A(\S^7)$.

The algebra $\A(\widetilde{S}^7_\theta)$ is generated by elements
$\{w_j, w_j^*, j=1,\dots,4\}$, the $w_j$'s being as in \eqref{transf} but with `coefficients' in
$\A(\SL_\theta(2,\IH))$. Clearly, $\sum_j w_j^* w_j = \rho^2$.
Then, the algebra of invariants of the action of $\SU(2)$ on $\A(\widetilde{S}^7_\theta)$ is generated by
\begin{eqnarray}\label{inflates4ins7}
&&
\widetilde{x}=\wuno \wunob  + \wdue \wdueb- \wtre \wtreb - \wqu \wqub \; ,
 \nn \\
&&
\ta=2(\wuno \wtreb + \wdue \wqub) \, , \qquad
\tb= 2(-\wuno \wqu + \wdue \wtre) \,,
\end{eqnarray}
together with $\rho^2$. This is so because the elements \eqref{inflates4ins7} are the images under the map $\Delta_L$ of the elements \eqref{s4ins7} that generate the algebra of invariants under the action of $\SU(2)$ on $\A(\S^7)$. This correspondence also gives for their commutation relations the same expressions as the ones in \eqref{4sphere} for the generators of $\A(\S^4)$. The difference is that we do not have the spherical relation of $\A(\S^4)$ any longer but rather we find that
\beq\label{infradius}
\widetilde{\alpha}^* \widetilde{\alpha} + \widetilde{\beta}^* \widetilde{\beta} + \widetilde x^2 = \big(\sum\nolimits_j w_j^* w_j\big)^2 = \rho^4.
\eeq
We denote by $\A(\widetilde S^4_\theta) \subset \A(\SL_\theta(2,\IH)) \ot \A(S^4_\theta) $ the algebra of invariants and conclude that the coaction $\Delta_L$ of $\A(\SL_\theta(2,\IH))$ on the $\SU(2)$ principal  fibration $\A(\S^4) \into \A(\S^7)$ generates a family of $\SU(2)$ principal  fibrations $\A(\widetilde S^4_\theta) \into \A(\widetilde S^7_\theta)$. Evaluating the central element $\rho^2$, for any $r \in \IR$ we get an $\SU(2)$ principal  fibration $\A(S_{\theta,r}^4) \into \A(S_{\theta,r}^7)$ of spheres of radius $r^2$ and $r$ respectively.

\bigskip
Motivated by the interpretation of the Hopf algebra $\SL_\theta(2,\IH)$ as a parameter space (see Sect.~\ref{se:family}), the coaction of $\A(\SL_\theta(2,\mathbb{H}))$ is extended to the forms $\Omega(S^4_\theta)$ by requiring that it commutes with $\dd$, {i.e.} $\Delta_L (\dd \omega) = (\id \otimes \dd ) \Delta_L(\omega)$, thus extending the differential of $\A(\S^4)$ to $\A(\SL_\theta(2,\IH)) \otimes \A(\S^4)$ as $(\id \otimes \dd)$.
Having these, we have the following characterization of $\A(\SL_\theta(2,\mathbb{H}))$ as conformal transformations. 
\begin{prop}
\label{prop:conf}
With $\ast_\theta$ the natural Hodge operator on $\S^4$, the algebra
$\A(\SL_\theta(2,\mathbb{H}))$ coacts by conformal transformations on $\Omega(S^4_\theta)$, that is
$$
\Delta_L ( *_\theta \omega)= (\id \ot *_\theta) \Delta_L(\omega) \;
, \qquad \forall \; \omega \in \Omega (S^4_\theta) \; .
 $$
\end{prop}
\begin{proof}
The map $\Delta_L$ is given by the classical coaction of $\A(\SL(2,\IH))$ on $\Omega(S^4)$ {\it as vector spaces} and only the two products on $\A(\SL(2,\IH))$ and $\Omega(S^4)$ are deformed.
Since $\ast_\theta$ coincides with the undeformed Hodge operator $\ast$ on $\Omega(S^4_\theta) \simeq \Omega(S^4)$ as vector spaces, the result follows from the fact that $\SL(2,\IH)$ acts by conformal transformations on $S^4$.
\end{proof}

\subsection{The quantum group $\SO_\theta(5,1)$}

By construction, the generators $\widetilde\alpha,\widetilde\beta,\widetilde x$ of $\A(\widetilde S^4_\theta)$ are the images under $\Delta_L$ of the corresponding $\alpha,\beta,x$ of $\A(\S^4)$. Some algebra yields
\begin{align*}
2\; \tx & = (a_1 \ba_1 + a_2 \ba_2 + b_1 \bb_1 + b_2 \bb_2 - c_1\bc_1 - c_2 \bc_2
- d_1 \bd_1 - d_2 \bd_2) \ot 1 \\
& ~+ (a_1 \ba_1 + a_2 \ba_2 - b_1 \bb_1 - b_2 \bb_2 - c_1\bc_1 - c_2 \bc_2
+ d_1 \bd_1 + d_2 \bd_2) \ot x \\
& ~+ (a_1 \bb_1 + \mu \, b_2 \ba_2 - c_1 \bd_1 - \mu \, d_2 \bc_2)  \otimes \alpha
+ (b_1 \ba_1 + \bar\mu \, a_2 \bb_2 - d_1 \bc_1 -
\bar\mu \, c_2 \bd_2)  \otimes \alpha^* \\
&  ~+ (a_1 \bb_{2} - \bar\mu \, b_1 \ba_2 - c_1 \bd_2 + \bar\mu \, d_1 \bc_2) \otimes  \beta + (b_2 \ba_1 - \mu \, a_2 \bb_1 - d_2 \bc_1 + \mu \, c_2 \bd_1) \otimes  \beta^*  \; ,
\end{align*}
\begin{align}\label{transf4}
\ta &= (a_1 \bc_1 +a_2 \bc_2 + b_1 \bd_1 + b_2 \bd_2) \otimes 1
+ (a_1 \bc_1 +a_2 \bc_2 - b_1 \bd_1 - b_2 \bd_2) \otimes x  \qquad\qquad \nn
\\
&  ~ + (a_1 \bd_1 + \mu \, b_2 \bc_2) \otimes \alpha
 + ( b_1 \bc_1+ \bar\mu \, a_2 \bd_2 ) \otimes \alpha^*   \nn
\\
& ~ + (a_1 \bd_2 - \bar\mu \, b_1 \bc_2) \otimes \beta + ( b_2 \bc_1- \mu \, a_2 \bd_1 ) \otimes  \beta^* \; ,
\end{align}
\begin{align*}
\tb &= (a_2 c_1 -a_1 c_2 + b_2 d_1 -b_1 d_2)  \otimes 1 +
(a_2 c_1 -a_1 c_2 - b_2 d_1 + b_1 d_2)  \otimes x   \qquad\qquad
\\
& ~ + (-a_1 d_2 + \mu \, b_2 c_1) \otimes \alpha
+ ( -b_1 c_2+ \bar\mu \, a_2 d_1) \otimes  \alpha^* \\
&~  + (a_1 d_1 - \bar\mu \, b_1 c_1) \otimes \beta + ( - b_2 c_2+\mu \, a_2 d_2 ) \otimes \beta^* \; .
\end{align*}
{}From the definition of $\rho$ in \eqref{def:rho}, using the commutation relations \eqref{commutus}, it follows that
\begin{align*} 2 \rho^2
& = \Delta_L ((u^*u)_{11} + (u^*u)_{22})
= \sum\nolimits_{ilk} (A_{il})^*A_{ik} \ot ((u_{l1})^* u_{k1} + (u_{l2})^* u_{k2}) \\
& = \sum\nolimits_{ilk} (A_{il})^*A_{ik} \ot \eta_{lk} (u_{k1} (u_{l1})^*
 + u_{k2} (u_{l2})^*) = \sum\nolimits_{ilk} (A_{il})^*A_{ik} \eta_{lk} \, \ot (u u^*)_{kl},
 \end{align*}
and, being $(u u^*)_{kl}$ the component $p_{kl}$ of the defining projector $p$ in \eqref{basicp},  an explicit computation yields
\begin{align} \label{rho-formula}
2 \rho^2 & = (a_1 \ba_1 + a_2 \ba_2 + c_1\bc_1 + c_2 \bc_2 + b_1 \bb_1 + b_2 \bb_2
+ d_1 \bd_1 + d_2 \bd_2) \ot 1 \nn \\
& ~+ (a_1 \ba_1 + a_2 \ba_2 + c_1\bc_1 + c_2 \bc_2 - b_1 \bb_1 - b_2 \bb_2
- d_1 \bd_1 - d_2 \bd_2) \ot x \nn \\
& ~+ (a_1 \bb_1 + \mu \, b_2 \ba_2 + c_1 \bd_1 + \mu \, d_2 \bc_2)  \otimes \alpha
+ (b_1 \ba_1 + \bar\mu \, a_2 \bb_2 + d_1 \bc_1 + \bar\mu \, c_2 \bd_2)  \otimes \alpha^* \nn \\
&  ~+ (a_1 \bb_{2} - \bar\mu \, b_1 \ba_2 + c_1 \bd_2 - \bar\mu \, d_1 \bc_2) \otimes  \beta + (b_2 \ba_1 - \mu \, a_2 \bb_1 + d_2 \bc_1 - \mu \, c_2 \bd_1) \otimes  \beta^*  \; .
\end{align}
In the expressions \eqref{transf4} and \eqref{rho-formula}, the elements of $\A(\SL_\theta(2,\IH))$ appear only quadratically. Rather than a coaction of $\A(\SL_\theta(2,\IH))$, on $\A(\S^4)$ there is a coaction of the $\IZ^2$-invariant subalgebra. We denote this by
$\A(\SO_\theta(5,1))$, a notation that will become clear presently.

In $\A(\IC^4_\theta)$, let us consider  the vector-valued function $X:=(r, x, \alpha, \alpha^*, \beta,
\beta^*)^t$, with  $r:=\uno
\unob + \due \dueb + \tre \treb + \qu \qub$ and
$x, \alpha, \beta$ are the quadratic elements, with the same formal expression as in \eqref{s4ins7}, but with the $z_\mu$'s in
$\A(\IC^4_\theta)$ (that is we do not impose any spherical relations).
A little algebra shows that they satisfy the condition
$$
\sum_{ij} g^{ij} X_i X_j = 0 \;, \qquad \mathrm{or} \qquad X^t g X = 0 \;,
$$
where $g$ is the metric on $\IR^6$ with signature $(5,1)$,
{i.e.} on $\IR^{5,1}$. In terms of the basis
$\{X_i\}_{i=1,\ldots,6}$ and with the natural identification of
$\IR^{5,1}$ with $\IR^{1,1} \oplus \IC^2$, this metric becomes
\beq\label{metricg}
g= \left( \half g^{ij} \right) =
\half \diag\left(
\begin{pmatrix} -2 & 0 \\ 0 & 2 \end{pmatrix},
\begin{pmatrix} 0 & 1 \\  1 & 0 \end{pmatrix},
\begin{pmatrix} 0 & 1 \\  1 & 0 \end{pmatrix}
\right) .
\eeq
The coaction in \eqref{aactc4} can be given on these quadratic
elements and  summarized by $\Delta_L(X_i) = \sum_j C_{ij} \otimes
X_j$. Here
the $C_{ij}$'s -- assembled in a matrix $C_\theta$ --, are
($\IZ^2$-invariant) elements in $\A(\SL_\theta(2,\IH))$ whose 
expression can be
read off from equations \eqref{transf4} and \eqref{rho-formula} simply
reading $r$ instead of $1$. Their commutation relations are obtained from the \eqref{algAtheta}:
$$
C_{il}C_{jm}= \nu_{ij}\nu_{ml}C_{jm}C_{il},
$$ 
where the matrix $\nu=(\nu_{ij})$ has entries all equal to $1$ except for 
$$
\nu_{35}= \nu_{46}= \nu_{54}=
\nu_{63}= \lambda \;, \quad \nu_{36}= \nu_{45}= \nu_{53}= \nu_{64}= \bar \lambda \;.
$$
There are two additional properties of the matrix $C_\theta$. 
The first one is that
\beq\label{cgc}
{C_\theta}^t \,g \,C_\theta =g \;,
\eeq
as we shall now prove. In order to simplify computations for this, we shall rearrange the generators and use, instead of $X$, the vector  
 $$
 Y = (\pi_{12} , \pi_{34}, \pi_{14}, \pi_{23},   \pi_{13}, \pi_{24}),
 $$
where the $\pi_{ij}$'s are the 2-minors of
the matrix $u$ in \eqref{utheta}: 
$$
\pi_{ij}:=u_{i1}u_{j2}- u_{i2}
u_{j1} \; ,  \qquad i<j \; , \;\;  i,j =1,\dots 4 \; .
$$
The relations with the $X_i$'s are 
\begin{eqnarray*}
X_1= Y_1 + Y_2 \; , & X_2= Y_1 - Y_2 \; , & X_3 =  2 Y_{3} \; , \\
X_4=  - 2 \mu\, Y_4 \; , & X_5= -2 Y_5  \; , &X_6= -2 \bar \mu\, Y_6 \;.
\end{eqnarray*}
On the generators $\pi_{ij}$'s, the coaction $\Delta_L$
in \eqref{left-coaction:theta} simply reads 
$$ 
\pi_{ij} \mapsto \Delta_L(\pi_{ij})=\sum_{\stackrel{l,s=1 \dots
4}{l<s}} {m_{ij}}^{ls} \ot \pi_{ls} \;,
$$
with the $m$'s given by the 2-minors of the matrix $A_\theta$:
\beq \label{minori}
{m_{ij}}^{ls} =
A_{il} A_{js} - \eta_{ls} A_{is} A_{jl} \;, \qquad l<s \; , i<j \;
.\eeq
With the generators $\pi_{ij}$'s the condition $X^t gX =0$ translates to
 \beq \label{pl}
\pi_{12}\pi_{34}+ \mu \,\pi_{14} \pi_{23} - \bar\mu \,\pi_{13}
\pi_{24}=0 \;,
\eeq
which, at the classical value of the deformation parameter, $\mu=\lambda=1$, 
is the Pl\"ucker quadric \cite{atiyah}. In turn, by rewriting the metric, this condition can be written as
\beq\label{metric}
Y^t h Y =0 \;, \quad \mathrm{with} \quad
h = \left( \half h^{IJ} \right) =
\half \diag \left(
\begin{pmatrix} 0 & 1 \\ 1 & 0 \end{pmatrix},
\begin{pmatrix} 0 & \mu \\  \mu & 0 \end{pmatrix},
\begin{pmatrix} 0 & -\bar \mu \\  -\bar \mu & 0 \end{pmatrix}
\right).
\eeq
Here and in the following we use capital letters to denote indices 
$$
I \in \{1=(12), \, 2=(34), \, 3=(14), \, 4=(23), \, 5=(13), \, 6=(24)\} .
$$

The statement in \eqref{cgc} that ${C_\theta}^t \,g \,C_\theta = g$ is equivalent to the 
following proposition whose proof is given in App.~\ref{Ap-B}.
\begin{prop}\label{prop:metric}
The minors in \eqref{minori} and the metric in \eqref{metric} satisfy the condition,
$$
\sum\nolimits_{I J}h^{IJ} ~ {m_I}^K ~{m_J}^L = h^{KL} .
$$
\end{prop}

The second relevant property of the matrix $C_\theta$ concerns its determinant. An element $\det(C_\theta)$ can be defined using a differential calculus, now on $\A(\IR_\theta^{5,1})$ (with relations dictated by those in $\A(\S^4)$ except for the spherical relation) as
$$
\Delta_L ( \dd X_1 \cdots \dd X_6 ) = \det(C_\theta) \otimes \dd X_1 \cdots \dd X_6.
$$
One expects that $\det(C_\theta)$ can be expressed in terms of $\det(A_\theta)$ as defined in \eqref{det} and that it {\it should indeed be equal to 1}. Instead of checking this via a direct computation, we observe that $\dd X_1 \cdots \dd X_6$ is a central element in the differential calculus on $\A(\IR^{5,1}_\theta)$ and has a classical limit which is invariant under the torus action. In the framework of the deformation described at the end of Sect.~\ref{se:principal} the result of $\Delta_L$ on it remains undeformed and coincides with the classical coaction of $\A(\SO(5,1))$ giving indeed $\det(C_\theta)=1$.

\begin{rem}
With the two properties above, the algebra $\A(SO_\theta(5,1))$ could have been defined  without reference to $\A(\SL_\theta(2,\IH))$. The entries of the matrix $C_\theta$ are its generators with relations derived as in Sect.~\ref{sse:SL2H} by imposing that $X_{i} \mapsto \sum_jC_{ij} \otimes X_j$ respects the commutation relations of $\A(\S^4)$, except the spherical relation.
In addition, one imposes the conditions $C^t_\theta g C_\theta = g$ and $\det(C_\theta)=1$.
Of course, this algebra is isomorphic to the Hopf subalgebra of $\IZ_2$-invariants in $\A(\SL_\theta(2,\IH))$ discussed above. It could also be obtained along the lines of \cite{Ri93a,varilly} (see also the beginning of Sect. \ref{se:SL2H}) by deforming the product on $\A(\SO(5,1))$
with respect to the adjoint action of the torus $\IT^2 \subset \SO(5,1)$.
\end{rem}

\section{A noncommutative family of instantons on $S^4_\theta$}
\label{se:family}

We mentioned in Sect.~\ref{se:principal} that out of the matrix valued function $u$ in \eqref{utheta} one gets a projection $p=u^*u$, given explicitly in \eqref{ptheta}, whose Grassmannian connection $\nabla=p \circ \dd $ has self-dual curvature: $*_\theta \nabla^2=\nabla^2$. The corresponding instanton connection 1-form -- acting on equivariant maps -- is expressed in terms of $u$ as well and it is an $\su(2)$-valued 1-form on $\S^7$.  Indeed,  the $\A(\S^4)$-module $\E$ determined by $p$
is isomorphic to the $\A(\S^4)$-module of equivariant maps for the defining representation $\pi$ of $\SU(2)$ on $\IC^2$:
$$
\E \simeq \A(S^7_\theta)\boxtimes_\pi \IC^2 := \left\{ f \in \A(S^7_\theta) \ot \IC^2 : (\id \otimes \pi(g)^{-1})(f)=(\alpha_g \otimes \id) (f) \right\},
$$
whose elements we write $f=\sum_a f_a \otimes e_a$ by means of the
standard basis $\{e_1, e_2\}$ of $\IC^2$.
The connection $\nabla=p \circ \dd : \E \to \E \otimes_{\A(S^4)} \Omega(\S^4)$ becomes on the equivariant maps:
$$
\nabla(f_a)= \dd f_a + \sum\nolimits_b \omega_{ab} f_b,\qquad a,b=1,2,
$$
where the connection 1-form $\omega=(\omega_{ab})$ is found to be given by
\beq\label{basic1form}
\omega_{ab} = \half \sum\nolimits_k
\big( (u^*)_{ak} \dd u_{kb} - \dd(u^*)_{ak} u_{kb} \big) .
\eeq
One has $\omega_{ab}=-(\omega^*)_{ba}$ and
$\sum_a \omega_{aa}=0$ so that $\omega$ is in
$\Omega^1(S^7_\theta) \otimes su(2)$.

Out of the coaction of the quantum group $\SL_\theta(2,\mathbb{H})$ on the Hopf fibration on $\S^4$, we shall get a family of such connections in the sense that we explain in the next sections.

\subsection{A family of  projections}
We shall first describe a family of vector bundles over $S^4_\theta$.  This is done by giving a family of suitable projections. We know from \eqref{transu} or
\eqref{left-coaction:theta} the transformation of the matrix $u$ to $\widetilde u$ for the coaction of $\A(\SL_\theta(2,\mathbb{H}))$:
$\widetilde u_{ia} = \Delta_L(u_{ia}) = \sum_j A_{ij} \otimes u_{ja}$, with $A_\theta=(A_{ij})$ the defining matrix of $\A(\SL_\theta(2,\mathbb{H}))$. The fact that the latter does not preserve the spherical relations is also the statement that
\beq \label{utnorm}
\sum\nolimits_k(\widetilde{u}^*)_{ak} \widetilde{u}_{kb} = \Delta_L\left(\sum\nolimits_k( u^*)_{ak}  u_{kb}\right) =  \Delta_L\left(\sum\nolimits_k z_k^* z_k \right) \delta_{ab} = \rho^2 \delta_{ab},
\eeq
or $ (\widetilde u)^* \widetilde u = \rho^2 \II_2$. Then, we define $P = (P_{ij}) \in \Mat_4(\A(\widetilde S^4_\theta))$ by
\beq \label{prho}
P :=  \widetilde u \;  \rho^{-2} \; (\widetilde u)^* , \qquad \mathrm{or} \quad
P_{ij}= \rho^{-2} \sum\nolimits_a \widetilde u_{ia} (\widetilde u^*)_{aj}.
\eeq
The condition $( \widetilde u)^* \widetilde u=\rho^2 \II_2$ gives that $P$ is an idempotent; being $*$-selfadjoint it is a projection.
\begin{rem}
For the above definition, we need to enlarge the algebra $\A(\widetilde S^4_\theta)$  by adding the extra element $\rho^{-2}$, the inverse of the positive self-adjoint central element $\rho^2$.  In fact, we shall also presently need the element $\rho^{-1}=\sqrt{\rho^{-2}}$. At the smooth level this is not problematic. The algebra $C^\infty(\widetilde S^4_\theta)$ can be defined as a fixed point algebra as in \cite{cd} and one finds that the spectrum of $\rho^2$ is positive and does not  contain the point $0$.
\end{rem}
Explicitly, one finds for the projection $P$ the expression
\begin{eqnarray*}
P &=& \frac{1}{2}\, \rho^{-2}
\begin{pmatrix}
 \rho^2+ \tx & 0 & \ta & \tb
\\ 0 &   \rho^2 + \tx & -\mu\, \tb^* & \bar{\mu}\, \ta ^* \\
\ta^* & -\bar{\mu}\, \tb &  \rho^2 - \tx &0\\ \tb^* & \mu\, \ta &0
&\rho^2 - \tx \end{pmatrix} ,
\end{eqnarray*}
a matrix strikingly similar to the matrix \eqref{ptheta} for the basic projection.
The entries of the projection $P$ are in $\A(\widetilde S^4_\theta)$, that is $\A(\SL_\theta(2,\IH)) \otimes \A(S^4_\theta)$:
we interpret $P$ as a {\it noncommutative family of projections} parametrized by the noncommutative space $\SL_\theta(2,\IH)$. This is the analogue for projections of the noncommutative families of maps that were introduced and studied in \cite{woro,soltan}.
The interpretation as a noncommutative family is justified by the classical case: at $\theta = 0$, there are evaluation maps $ev_x : \A(\SL(2,\IH)) \to \IC$ and for each point
$x$ in $\SL(2,\IH)$, $(ev_x \otimes \id) P$ is a projection in $\Mat_4(\A(S^4))$, that is a bundle over $S^4$. Although there need not be enough evaluation maps available in the noncommutative case, we can still work with the whole family at once.

As mentioned, we think of the Hopf algebra $\SL_\theta(2,\IH)$ as a parameter space and we extend to $\A(\SL_\theta(2,\IH)) \otimes \A(\S^4)$ the differential of $\A(\S^4)$ as $(\id \otimes \dd)$ (and similarly, the Hodge star operator of $\A(\S^4)$ as $(\id \otimes \ast_\theta)$).  Having these, out of the projection $P$ one gets a noncommutative family of instantons.
\begin{prop}
The family of connections
$\widetilde \nabla := P \circ (\id \otimes \dd)$ has self-dual curvature 
$\widetilde\nabla^2=P ((\id \otimes \dd)P)^2$, that is,
$$
(\id \ot \ast_\theta) P ((\id \otimes \dd)P)^2 = P ((\id \otimes \dd)P)^2.
$$
\end{prop}
\begin{proof}
{}From Proposition~\ref{prop:conf} we know that $\A(\SL_\theta(2,\IH))$ coacts by conformal transformations and the curvature $\widetilde\nabla^2=P((\id \otimes \dd)P)^2$ is
the image of the curvature $p(\dd p)^2$ under the coaction of $\A(\SL_\theta(2,\IH))$.
\end{proof}
It was shown in \cite{cl} that the charge of the basic instanton $p$ is $1$. This charge was given as a pairing between the second component of the Chern character of $p$ -- an element in the cyclic homology group $\HC_4(\A(\S^4))$ -- with the fundamental class of $\S^4$ in the cyclic cohomology $\HC^4(\A(\S^4))$. The zeroth and first components of the Chern character were shown to vanish identically in $\HC_0(\A(\S^4))$ and $\HC_2(\A(\S^4))$, respectively.
We will reduce the computation of the Chern character for the family of projections $P$ to this case by proving that $P$ is equivalent to the projection $1\ot p$. Hence, we conclude that $P$ represents the same class as $1 \otimes p$ in the K-theory of the algebra $\A(\SL_\theta(2,\IH)) \otimes \A(\S^4)$.

Recall that two projections $p,q$ are {\it Murray-von Neumann equivalent} if there exists a partial isometry $V$ such that $p=V V^*$ and $q=V^* V$.
\begin{lemma}
The projection $P$ is Murray-von Neumann equivalent to the projection
$1 \otimes p$ in the algebra
$M_4 \left( \A(\SL_\theta(2,\IH)) \otimes \A(\S^4)\right)$.
\end{lemma}
\begin{proof}
Define the matrix $V = (V_{ik}) \in M_4 \left( \A(\SL_\theta(2,\IH)) \otimes \A(\S^4)\right)$ by
\begin{align*}
V_{ik} = \rho^{-1} A_{ij} \otimes p_{jk}=\rho^{-1} A_{ij} \otimes u_{ja} (u^*)_{ak} =
\rho^{-1}\widetilde{u}_{ia} (1\otimes (u^*)_{ak}),
\end{align*}
with $\tilde u = \left(\tilde u_{ia}\right)$ as in \eqref{transu}. For its adjoint, we have
$$
(V^*)_{ik} =  \rho^{-1}(1\otimes u_{ia})(\widetilde{u}^*)_{ak}.
$$
Then, using \eqref{utnorm}, one obtains
\begin{align*}
(V^* V) _{il} & = \sum\nolimits_k(V^*)_{ik} V_{kl} = \rho^{-2}
\sum\nolimits_{kab} (1\otimes u_{ia}) (\widetilde{u}^*)_{ak}
\widetilde{u}_{kb} (1\otimes (u^*)_{bl}) \\
& = \rho^{-2}\sum\nolimits_{ab} (1\otimes u_{ia}) (\rho^2 \delta_{ab}) (1\otimes (u^*)_{bl}) = 1 \otimes \sum\nolimits_a u_{ia}(u^*)_{al} = 1 \otimes p_{il},
\intertext{and}
(VV^*) _{il} & = \sum\nolimits_k V_{ik} (V^*)_{kl} =
\rho^{-2} \sum\nolimits_{kab}\widetilde{u}_{ia} (1\otimes (u^*)_{ak})
(1\otimes u_{kb})(\widetilde{u}^*)_{bl}
\\
& =\rho^{-2} \sum\nolimits_{kab} \widetilde{u}_{ia}(1\otimes  (u^*)_{ak} u_{kb})(\widetilde{u}^*)_{bl} = \rho^{-2} \sum\nolimits_{ab} \widetilde{u}_{ia}(1\otimes  \delta_{ab})(\widetilde{u}^*)_{bl}  \\
& = \rho^{-2} \sum\nolimits_a \widetilde{u}_{ia}(\widetilde{u}^*)_{al} =P_{il},
\end{align*}
which finishes the proof.
\end{proof}
\noindent
It follows from this lemma that the components $\ch_n(P) \in \HC_{2n}(\A(\SL_\theta(2,\IH)) \otimes \A(\S^4))$, with $n=0,1,2$, of the Chern character of $P$ coincide with the pushforwards $\phi_* \ch_n(p)$ of $\ch_n(p) \in \HC_{2n}( \A(\S^4)) $ under the algebra map
$$
\phi : \A(\S^4) \to \A(\SL_\theta(2,\IH)) \otimes \A(\S^4),\quad a \mapsto 1 \otimes a.
$$
As a consequence, both $\ch_0(P)$ and $\ch_1(P)$ are zero since $\ch_0(p)$ and $\ch_1(p)$ 
vanish \cite{cl}.
Next, we would like to compute the charge of the family of instantons by pairing $\ch_2(P)$ with the fundamental class $[\S^4] \in \HC^4( \A(\S^4))$; classically, this corresponds to an integration over $S^4$ giving a value 1 of the charge which is constant over $\SL(2,\IH)$. As said, the Chern character $\ch_2(P)$ is an element in $\HC_{4}\left(\A(\SL_\theta(2,\IH)) \otimes \A(\S^4)\right)$ which at first sight seems unsuitable to pair with an element in $\HC^4( \A(\S^4))$. However, there is a pairing between the K-theory group $\K_0\left(\A(\SL_\theta(2,\IH)) \otimes \A(\S^4 ) \right)$ and the K homology group $\K^0( \A(\S^4))$. Recall that for any algebra $\A$ from Kasparov's KK-theory one has that $\K^0(\A) = \KK(\A, \IC)$ and $\K_0( \A) = \KK( \IC, \A)$. As described in \cite[App. IV.A]{connes}, for algebras $\A,\B$ and $\C$, there is a map $\tau_C : \KK(\A,\B) \to \KK(\C \otimes \A, \C \otimes \B)$ which simply tensors a Kasparov $\A-\B$ module by $\C$ on the left. In our case we get an element $\tau_{\A(\SL_\theta(2,\IH))} [\S^4] \in \KK\left( \A(\SL_\theta(2,\IH)) \otimes \A(\S^4) , \A(\SL_\theta(2,\IH)) \right)$ which can be paired with $[P]$ via the cup product \cite{Kas80}. Thus we obtain the desired pairing: 
\begin{align*}
\KK\left( \IC ,  \A(\SL_\theta(2,\IH))  \otimes \A(\S^4) \right) & \times \KK \left( \A(\S^4) , \IC \right)   \to \KK \left( \IC ,  \A(\SL_\theta(2,\IH))\right),  \\ \left(  [P]  , [\S^4]  \right) & \mapsto \langle [P]  , \tau_{\A(\SL_\theta(2,\IH))}[\S^4] \rangle . 
\end{align*}
Having $[P] = [1 \otimes p]$, we obtain
$$
\langle [P] , \tau_{\A(\SL_\theta(2,\IH))}[\S^4] \rangle =  [1]  \otimes \langle [p]  , [\S^4]  \rangle = [1] \in \K_0(  \A(\SL_\theta(2,\IH))) ,
$$
where in the last line we used the equality $\langle [p]  , [\S^4]  \rangle = 1 $ proved in \cite{cl} too. The above is the statement that the value 1 of the topological charge is constant over the family. 

\subsection{A family of connections}

When transforming $u$ by the coaction of $\SL_\theta(2,\IH)$ in \eqref{transu}, one transforms the connection 1-form $\omega$ in \eqref{basic1form} as well to $\widetilde\omega = (\widetilde \omega_{ab})$ with,
\beq\label{transomega}
\widetilde \omega_{ab} := \Delta_L (\omega_{ab}) =  \half \sum\nolimits_{kij}
(A^*)_{ik} A_{kj} \otimes \big( (u^*)_{ai} \dd u_{jb} - \dd(u^*)_{ai} u_{jb} \big).
\eeq
Since $\Delta_L$ is linear, $\widetilde\omega$ is still traceless
($\sum_a \widetilde\omega_{aa}=0$) and skew-Hermitean
($\widetilde\omega_{ab}=-(\widetilde \omega^*)_{ba}$).

\begin{prop}
The instanton connection 1-form $\omega$ is invariant under the coaction of the quantum group $\Sp_\theta(2)$, that is for this quantum group one has
$$
\Delta_L(\omega_{ab})=1 \otimes \omega_{ab} .
$$
\end{prop}
\begin{proof}
This is a simple consequence of the fact that for $\Sp_\theta(2)$ one has $\sum_k  (A^*)_{ik} A_{kj} = \delta_{ij}$ from which \eqref{transomega} reduces to
$\widetilde\omega_{ab} = 1 \otimes \omega_{ab}$.
\end{proof}

Hence, the relevant space that parametrizes the connection one-forms is not $\SL_\theta(2,\IH)$ but rather the quotient of $\SL_\theta(2,\IH)$ by $\Sp_\theta(2)$. Denoting by $\pi$  the natural quotient map from $\A(\SL_\theta(2,\IH))$ to $\A(\Sp_\theta(2))$,  the algebra of the quotient is the algebra of coinvariants of the natural left coaction
$\Delta_L = (\pi \otimes \id ) \circ \Delta$
of $\Sp_\theta(2)$ on $\SL_\theta(2,\IH)$:
$$
\A(\M_\theta):= \{ a \in \A(\SL_\theta(2,\IH)) ~|~ \Delta_L(a) = 1\otimes a \}.
$$
Since $\Sp_\theta(2)$ is a quantum subgroup of $\SL_\theta(2,\IH)$ the quotient is well defined: the algebra $\A(\M_\theta)$ is a quantum homogeneous space and the inclusion $\A(\M_\theta) \into \A(\SL_\theta(2,\IH))$ is a noncommutative principal bundle with $\A(\Sp_\theta(2))$ as structure group.

\begin{lemma}
The quantum quotient space $\A(\M_\theta)$ is generated as an algebra by the elements $m_{ij}:=\sum_k (A^*)_{ik} A_{kj} = \sum_k (A_{ki})^* A_{kj}$.
\end{lemma}
\begin{proof}
Since the relations in the quotient $\A(\Sp_\theta(2))$ are quadratic in the matrix elements $A_{ij}$ and $(A_{ij})^*$, the generators of $\A(\M_\theta)$ have to be at least quadratic in them. For the first leg of the tensor product $\Delta(a)$ to involve these relations in $\A(\Sp_\theta(2))$, we need to take $a=\sum_i (A_{ik})^* A_{il}$, so that
\begin{align*}
(\pi\otimes \id) \Delta(a)
& = \sum\nolimits_{imn} \pi((A_{im})^* A_{in}) \otimes (A_{mk})^* A_{nl} \nn \\
& = \sum\nolimits_{imn} \pi((A_{im})^*) \pi(A_{in}) \otimes (A_{mk})^* A_{nl}
= \sum\nolimits_{mn} \delta_{mn} \otimes (A_{mk})^* A_{nl},
\end{align*}
giving the desired result.
\end{proof}

We will think of the transformed $\widetilde\omega$ in \eqref{transomega}
as {\it a family of connection one-forms parametrized by the noncommutative space $\M_\theta$}.  At the classical value $\theta = 0$, we get the moduli space $\M_{\theta=0}=\SL(2,\IH)/\Sp(2)$ of instantons of charge 1. For each point $x$ in $\M_{\theta=0}$, the evaluation map $ev_x : \A(\M_{\theta=0}) \to \IC$ gives an instanton connection (i.e. one with self-dual curvature) $(ev_x \otimes \id) \widetilde\omega$ on the bundle over $S^4$ described by $(ev_x \otimes \id) P$.

\subsection{The space $\mathcal{M}_\theta$ of connections and its geometry}\label{msconn}

The structure of the algebra $\A(\M_\theta)$ is deduced from that of $\A(\SL_\theta(2,\IH))$. We collect the generators
$m_{ij}=\sum_k (A_{ki})^* A_{kj}$ into a matrix $M:=(m_{ij})$.
Explicitly, one finds
\beq\label{defim}
M =
\begin{pmatrix}
m & 0 & g_1 & \bg_2
\\
0 & m & -\bar{\mu} \; g_2 & \mu \; \bg_1
\\
\bg_1 & -\mu \; \bg_2 & n & 0
\\
g_2 & \bar\mu \; g_1 & 0 & n
\end{pmatrix}
\eeq
with its entries related to those of the defining matrix $A_\theta$ in \eqref{Atheta} of $\A(\SL_\theta(2,\IH))$ by 
\begin{align}\label{mlk}
m = m^* &= \ba_1 a_1 + \ba_2 a_2 + \bc_1 c_1 + \bc_2 c_2 \; ,   \\
n = n^* &= \bb_1 b_1 + \bb_2 b_2 + \bd_1 d_1 + \bd_2 d_2  \; , \nn \\
g_1&= \ba_1 b_1 + \mu\,  \bb_2 a_2 + \bc_1 d_1 + \mu\,  \bd_2 c_2 \; ,  \nn \\
g_2 & =  \bb_2 a_1  - \mu\, \ba_2 b_1 + \bd_2 c_1 - \mu\, \bc_2 d_1 \nn
\; .
\end{align}
As for the commutation relations, one finds that both $m$ and $n$ are central:
\begin{subequations}\label{rel:M}
\beq
m~x=x~m , \quad  n~x=x~n \;  \quad \forall \; x \in \mathcal{M}_\theta ;
\eeq
that $g_1$ and $g_2$ are normal:
\beq
g_1 \bg_1 = \bg_1 g_1 , \quad g_2 \bg_2 = \bg_2 g_2 ;
\eeq
and that
\beq
  g_1 g_2 = {\mu}^2 g_2 g_1 , \quad
  g_1 \bg_2 = \bar{\mu}^2 \bg_2 g_1 .
\eeq
\end{subequations}
together with their conjugates. There is also a quadratic relation,
\beq\label{rel:M:det}
m n  - ( \bg_1 g_1 + \bg_2 g_2) = 1,
\eeq
coming from the condition $\det(A_\theta)=1$.
Indeed, one first establishes that besides the product $mn$,  also
$\bg_1 g_1 + \bg_2 g_2$ is a central
element in $\A(\mathcal{M}_\theta)$, and then computes
\begin{align*}
m n - \bg_1 g_1 + \bg_2 g_2
& =  \ba_1 a_1 \bd_1 d_1 + \ba_1 a_1 \bd_2 d_2 +
\ba_2 a_2 \bd_1 d_1 + \ba_2 a_2 \bd_2 d_2 +
\bb_1 b_1 \bc_1 c_1 + \bb_1 b_1 \bc_2 c_2
\\
& \quad + \bb_2 b_2 \bc_1 c_1 + \bb_2 b_2 \bc_2 c_2
-\ba_1 c_1 \bd_1 b_1 - \ba_1 c_1 \bd_2 b_2
-\ba_2 c_2 \bd_1 b_1 - \ba_2 c_2 \bd_2 b_2
\\
& \quad - \bb_1 d_1 \bc_1 a_1 - \bb_1 d_1 \bc_2 a_2
-\bb_2 d_2 \bc_1 a_1 - \bb_2 d_2 \bc_2 a_2
-\ba_1 \bc_2 d_2 b_1 + \ba_1 \bc_2  d_1 b_2
\\
& \quad + \ba_2 \bc_1  d_2 b_1
-\ba_2 \bc_1 d_1 b_2 - \bb_1 \bd_2  c_2 a_1
+\bb_1 \bd_2 c_1 a_2 + \bb_2 \bd_1  c_2 a_1
-\bb_2 \bd_1 c_1 a_2
\\
& = \det(A_\theta)
\end{align*}
by a direct comparison with the expression \eqref{det}.  Elements $(m_{ij})$ of the matrix $M$ enter the expression for $\rho^2$.
With $p_{kl}$ the components of the defining projector $p$ in \eqref{basicp}  and having formula \eqref{rho-formula}, one finds that
\begin{align*}
\rho^2 &= \half \sum\nolimits_{ij}
 \eta_{ij}\, m_{ij} \ot p_{ji}  \\
& = \half \left[(m+n)\ot 1 +  (m-n)\ot x  + \mu\,  g_1^* \ot \alpha +
  \bar\mu\, g_2 \ot \beta + \bar\mu\, g_1 \ot \alpha^* + \mu\, \bg_2 \ot \beta^* \right].
\end{align*}
In particular, for $A_{ij} \in \A(\Sp_\theta(2))$ one gets
$\rho^2= \frac{1}{2} (1 \ot \tr(p))= 1 \ot 1$,
as already observed in Remark~\ref{rem:rho}.

\subsection{The boundary of $\mathcal{M}_\theta$}
The defining matrix $M$ of $\mathcal{M}_\theta$ in \eqref{defim}, with the commutation relations among its entries, is strikingly similar to the defining projection $p$ of $\A(\S^4)$  in \eqref{ptheta} with the corresponding commutation relations. Clearly, the crucial difference is that while for $\A(\S^4)$ we have a spherical relation, for $\mathcal{M}_\theta$ we have the relation
\eqref{rel:M:det} which makes $\mathcal{M}_\theta$ a $\theta$-deformation of a hyperboloid in 6 dimensions. This becomes more clear if we introduce two central elements $w$ and $y$, given in terms of $m,n$ by
$$
w:=\half(m+n); \qquad y:=\half(m-n) .
$$
Relation \eqref{rel:M:det} then reads
\beq\label{relbis}
w^2 - (y^2 + \bg_1 g_1 + \bg_2 g_2) = 1,
\eeq
making evident the hyperboloid structure.
Let us examine its structure at `infinity'. We first adjoin the inverse
of $w$ to $\A(\mathcal{M}_\theta)$, and stereographically project
onto the coordinates,
$$
Y := w^{-1} y, \quad G_1 := w^{-1} g_1,  \quad
G_2 := w^{-1} g_2 .
$$
The relation \eqref{relbis} becomes,
$$
Y^2 + G^*_1 G_1 + G^*_2 G_2 = 1 - w^{-2} .
$$
Evaluating $w$ as a real number, and taking its `limit to
infinity' we get a spherical relation,
$$
Y^2 + G^*_1 G_1 + G^*_2 G_2 = 1 .
$$
By combining this with relations \eqref{rel:M}, we can conclude that at the
`boundary' of $\mathcal{M}_\theta$, we re-encounter the noncommutative 4-sphere $\A(S^4_{\theta})$ via the identification
$$
Y \leftrightarrow x \; , \qquad G_1 \leftrightarrow \alpha \;
\qquad G_2 \leftrightarrow \beta.
$$
The above construction is the analogue of the classical  structure,
in which 4-spheres are found at the boundary of the moduli space.

\section{Outlook}
We have constructed a noncommutative family of instantons of charge 1 on the noncommutative 4-sphere $\S^4$. The family is parametrized by a noncommutative space $\M_\theta$ which reduces to the moduli space of charge 1 instantons on $S^4$ in the limit when $\theta \to 0$. Although this means that $\M_\theta$ is a quantization of the moduli space $\M_{\theta=0}$, it does not imply that it is itself a space of moduli. In order to call this the moduli space of charge 1 instantons on $\S^4$ a few things must be clarified.  We mention in particular two important points that for the moment lack a proper understanding.

First of all, we are confronted with the difficulty of finding a proper notion of gauge group and gauge transformations. A naive dualization of the undeformed  construction would lead one to consider  the group of $\A(\SU(2))$-coequivariant algebra maps from the algebra $\A(\SU(2))$ to $\A(\S^7)$, equipped with the convolution product. However, since the algebra $\A(\SU(2))$ is commutative as opposed to $\A(\S^7)$, one quickly realizes that there are not so many elements in this group (an interesting open problem is to find in general a correct noncommutative analogue of the group of maps from a space $X$ to a group $G$).

The second open problem is related to the fact that one would need some sort of universality for the noncommutative family of instantons to call it a moduli space. A possible notion of universality could be defined as follows. A family of instantons parametrized by $\A(\M)$ is said to be {\it universal} if for any other noncommutative family of instantons parametrized by, say, an algebra $\B$, there exists an algebra map $\phi: \M \to \B$ such that this family can be obtained from the universal family via the map $\phi$. Again, this is the analogue of the notion of universality for noncommutative families of maps as in \cite{woro,soltan}. But it appears that, in order to prove universality for the actual family that we have constructed in the present paper, an argument along the classical lines -- involving a local construction of the moduli space from its tangent bundle \cite{ahs} -- fails here,  due to the fact that there is no natural notion of a tangent space to a noncommutative space.

Progress on both of these problems must await another time.

\subsection*{Acknowledgments}  
We thank Simon Brain, Eli Hawkins, Mark Rieffel, Lech Woronowicz  and Makoto Yamashita for useful discussions and remarks. G.L. and C.R. were partially supported by the `Italian project PRIN06 - Noncommutative geometry, quantum groups and applications'.
 C.P. gratefully acknowledges  support from MRTN-CT 2003-505078, INDAM, MKTD-CT 2004-509794 and SNF. 


 \appendix

\section{Explicit commutation relations}\label{Ap-acr}
For convenience, we list the explicit commutation relations \eqref{algAtheta} of the elements of the matrix \eqref{Atheta}.
The not trivial ones are the following,
 \begin{align*} 
 a_1 b_1 = \bar\mu ~b_1 a_1  \qquad\quad a_2 b_1 = \mu ~b_1 a_2
 \qquad\quad a_1 c_1 = \mu ~c_1 a_1  \qquad\quad a_2 c_1 = \mu ~c_1 a_2    \\
 a_1 b_2 = \mu ~b_2 a_1        \qquad\quad   a_2 b_2 = \bar \mu ~b_2 a_2
 \qquad\quad  a_1 c_2 = \mu ~c_2 a_1 \qquad\quad a_2 c_2 = \mu ~c_2 a_2   \\
 a_1 \bb_1 = \mu ~\bb_1 a_1      \qquad\quad   a_2 \bb_1 = \bar \mu ~\bb_1 a_2 \qquad\quad  a_1 \bc_1 = \bar \mu ~\bc_1 a_1 \qquad\quad a_2 \bc_1 = \bar \mu ~\bc_1 a_2  \\
 a_1 \bb_2 = \bar\mu ~\bb_2 a_1  \qquad\quad   a_2 \bb_2 = \mu ~\bb_2 a_2 \qquad\quad  a_1 \bc_2 = \bar\mu ~\bc_2 a_1 \qquad\quad  a_2 \bc_2 = \bar \mu ~\bc_2 a_2  
\end{align*}
\begin{align*}
 a_1 d_1 =  d_1 a_1 \qquad\quad   a_2 d_1 = \mu^2 ~d_1 a_2 
\qquad\quad   b_1 c_1 = \mu^2 ~c_1 b_1 \qquad\quad   b_2 c_1 = c_1 b_2 \\
 a_1 d_2 = \mu^2 ~d_2 a_1 \qquad\quad   a_2 d_2 =  d_2 a_2
  \qquad\quad b_1 c_2 = c_2 b_1  \qquad\quad   b_2 c_2 = \mu^2 ~c_2 b_2 \\
 a_1 \bd_1 =  \bd_1 a_1 \qquad\quad   a_2 \bd_1 = {\bar\mu}^2 ~ \bd_1 a_2 
\qquad\quad  b_1 \bc_1 = {\bar\mu}^2 ~\bc_1 b_1  \qquad\quad   b_2 \bc_1 = \bc_1 b_2 \\
 a_1 \bd_2 = {\bar\mu}^2 ~\bd_2 a_1  \qquad\quad   a_2 \bd_2 = \bd_2 a_2
 \qquad\quad  b_1 \bc_2 =  \bc_2 b_1  \qquad\quad   b_2 \bc_2 = {\bar\mu}^2 ~\bc_2 b_2 
 \end{align*}
\begin{align*}
  b_1 d_1 = \mu ~d_1 b_1  \qquad\qquad b_2 d_1 = \mu ~d_1 b_2 \qquad\quad   c_1 d_1 = \bar\mu ~d_1 c_1
  \qquad\quad   c_2 d_1 = \mu ~d_1 c_2 \\
  b_1 d_2 = \mu ~d_2 b_1 \qquad\qquad b_2 d_2 = \mu ~d_2 b_2 \qquad\quad   c_1 d_2 = \mu ~d_2 c_1
  \qquad\quad   c_2 d_2 = \bar \mu ~d_2 c_2 \\
   b_1 \bd_1 = \bar\mu ~\bd_1 b_1 \qquad\qquad
 b_2 \bd_1 = \bar \mu ~\bd_1 b_2 \qquad\quad  c_1 \bd_1 = \mu ~\bd_1 c_1      \qquad\quad   c_2 \bd_1 = \bar \mu ~\bd_1 c_2\\
 b_1 \bd_2 = \bar\mu ~\bd_2 b_1 \qquad\qquad
 b_2 \bd_2 = \bar\mu ~\bd_2 b_2 \qquad\quad  c_1 \bd_2 = \bar\mu ~\bd_2 c_1  \qquad\quad  c_2 \bd_2 = \mu ~\bd_2 c_2
 \end{align*}
 together with their conjugates.

\section{Explicit proof of Proposition~\ref{prop:metric}}\label{Ap-B}
We prove here that the minors in \eqref{minori} and the metric in \eqref{metric} satisfy the condition 
\beq\label{proB}
\sum\nolimits_{I J}h^{IJ} ~ {m_I}^K ~{m_J}^L = h^{KL} 
\eeq
of Proposition~\ref{prop:metric}. As said in the main text, this is equivalent to the fact that the defining matrix $C_\theta$ of $\A(\SO_\theta(5,1))$ satisfy ${C_\theta}^t \,g \,C=g$ with $g$ the metric in \eqref{metricg}.

We first prove \eqref{proB} when in the right hand side $h^{KL} \neq 0$, namely for the cases $(K,L)=(1,2), \,(3,4), \, (5,6)$. For this we
need the formula \eqref{det-cond} for the determinant of $A_\theta$. A little algebra show that the determinant can also be written as
\beq \label{det2} 
\det (A_\theta)=\sum_{\sigma \in S_4}
(-1)^{|\sigma|} {\bar \varepsilon}^{\sigma}
A_{\sigma(1),1}A_{\sigma(2),2}A_{\sigma(3),3}A_{\sigma(4),4} 
\eeq
where ${\bar \varepsilon}^{\sigma}=\overline{\varepsilon^{\sigma}}$. 
Also, for the tensor $\varepsilon$
we find relations: 
\beq\label{eps} 
\varepsilon^{ijkl}= \eta_{ji}
\varepsilon^{jikl} \quad ; \quad \varepsilon^{ijkl}= \eta_{lk}
\varepsilon^{ijlk} \quad ; \quad \varepsilon^{ijkl} =\eta_{kj}
\varepsilon^{ikjl} \; ,
\eeq
and analogues for $\bar \varepsilon$, 
\beq\label{epsbar}
{\bar
\varepsilon}^{ijkl}= \eta_{ij} {\bar \varepsilon}^{jikl} \quad ;
\quad {\bar \varepsilon}^{ijkl}= \eta_{kl} {\bar
\varepsilon}^{ijlk} \quad , \quad {\bar\varepsilon}^{ijkl}
=\eta_{jk} {\bar \varepsilon}^{ikjl} \; .
\eeq
Given $\sigma \in S_4$, we let $\sigma '= (12) \sigma$ and $\sigma ''=
(34) \sigma$, and compute,
\begin{align}\label{det12}
\det (A_\theta)&=\sum_{\sigma \in S_4} (-1)^{|\sigma|} {\bar \varepsilon}^{\sigma}
A_{\sigma(1),1}A_{\sigma(2),2}A_{\sigma(3),3}A_{\sigma(4),4}
\nn \\
&=\sum_{\sigma \in S_4 \setminus \sigma'} (-1)^{|\sigma|} \left( {\bar \varepsilon}^{\sigma}
 A_{\sigma(1),1}A_{\sigma(2),2}- {\bar \varepsilon}^{\sigma'} A_{\sigma(2),1}A_{\sigma(1),2}\right)
A_{\sigma(3),3}A_{\sigma(4),4}
\nn \\
&=\sum_{\sigma \in S_4 \setminus \sigma'} (-1)^{|\sigma|} \left( {\bar \varepsilon}^{\sigma}
 A_{\sigma(1),1}A_{\sigma(2),2}-  {\bar
   \varepsilon}^{\sigma'}\eta_{\sigma(1)\sigma(2)} \eta_{12} A_{\sigma(1),2}A_{\sigma(2),1}\right)
A_{\sigma(3),3}A_{\sigma(4),4}
\nn \\
&=\sum_{\sigma \in S_4 \setminus \sigma'} (-1)^{|\sigma|} {\bar \varepsilon}^{\sigma}
{m_{\sigma(1)\sigma(2)}}^{12} ~
A_{\sigma(3),3}A_{\sigma(4),4}
\nn \\
&=\sum_{\sigma \in S_4 \setminus \{\sigma' , \sigma ''\}} (-1)^{|\sigma|}
{\bar \varepsilon}^{\sigma}
 ~{m_{\sigma(1)\sigma(2)}}^{12} ~
(A_{\sigma(3),3}A_{\sigma(4),4} - \eta_{34}A_{\sigma(3),4}A_{\sigma(4),3})
\nn \\
&=\sum_{\sigma \in S_4 \setminus \{\sigma' \sigma ''\}} (-1)^{|\sigma|}
{\bar \varepsilon}^{\sigma} ~
{m_{\sigma(1)\sigma(2)}}^{12} ~
{m_{\sigma(3)\sigma(4)}}^{34} .
\end{align}
Since the ${m_{ij}}^{kl}$
where defined for $i<j, ~k<l$, we choose  $\sigma\in S_4 \setminus 
\{\sigma' \sigma ''\} $ such that $\sigma_1 < \sigma_2$ and $\sigma_3 < \sigma_4$.
Hence the sum above runs over $\sigma=\left((\sigma_1,\sigma_2), (\sigma_3, \sigma_4)\right) \in {\mathcal I} $, where
\begin{multline*}
{\mathcal I} :=\left\{ \Big((1,2),(3,4) \Big); \Big((1,3),(2,4) \Big); \Big((1,4),(2,3) \Big); \right.\\
\left. \Big((2,3),(1,4) \Big); \Big((2,4),(1,3) \Big); \Big((3,4),(1,2) \Big) \right\}.
\end{multline*}
Finally, using the explicit form of the ${\bar \varepsilon}$'s
the above formula \eqref{det12} reads
$$
\det (A_\theta) = \sum\nolimits_{IJ} h^{IJ} {m_I}^1 {m_J}^2, 
$$
and the condition $\det (A_\theta)=1$ proves \eqref{proB} for $K=1, L=2$. The relation above coincides with the 'hyperboloid' relation \eqref{rel:M:det} for the generators of the matrix $\mathcal{M}_\theta$. 

For the other two cases we use different orders for the $A_{\sigma(i)i}$ in
\eqref{det2}. Similar procedures to the one in \eqref{det12} -- and using the properties \eqref{eps} and \eqref{epsbar} -- lead to
$$
\det (A_\theta) =  \eta_{24} \sum_{\sigma \in S_4}
(-1)^{|\sigma|}\, {\bar
\varepsilon}^{\sigma}~ {m_{\sigma(1)\sigma(4)}}^{14} ~ {m_{\sigma(2)\sigma(3)}}^{23}
 \;,
$$
which gives 
$$
\mu \, \det (A_\theta) = \sum\nolimits_{IJ} h^{IJ} {m_I}^3 {m_J}^4 , 
$$
hence proving \eqref{proB} for $K=3, L=4$;  and to
$$
\det (A_\theta)=
- \eta_{23} \, \sum_{\sigma \in S_4 }
(-1)^{| \sigma |} \, {\bar
\varepsilon}^{\sigma} ~
{m_{\sigma(1)\sigma(3)}}^{13} ~ {m_{\sigma(2)\sigma(4)}}^{24} \;,
$$
which gives
$$
\bar  \mu \, \det (A_\theta) = - \sum\nolimits_{IJ} h^{IJ} {m_I}^5 {m_J}^6 ,
$$
hence proving \eqref{proB} for $K=5, L=6$. 

\bigskip
Finally, we have \eqref{proB} when $h^{KL} = 0$ in the right hand side; for these cases \eqref{proB} is  
$$
{m_{12}}^{ij}  ~ {m_{34}}^{kl} + {m_{34}}^{ij}  ~ {m_{12}}^{kl}
+ \mu\, {m_{23}}^{ij}  ~ {m_{14}}^{kl} +  \mu\, {m_{14}}^{ij}  ~
{m_{23}}^{kl} - \bar\mu\, {m_{24}}^{ij}  ~ {m_{13}}^{kl} - \bar\mu \, {m_{13}}^{ij}  ~
{m_{24}}^{kl} =0 .
$$
These can be proved with the explicit expressions  of the
${m_{ij}}^{kl}$ in \eqref{minori} and observing that the hypothesis
$I=(ij), K=(kl)$ such that $h^{IK} =0 $ implies four possibilities: $i=k$, 
$i=l$, $j=k$ or $j=l$. In addition, the relation
${m_{ij}}^{kl}= -\eta^{kl}{m_{ij}}^{lk}$ reduces the computations to just
one case, say $i=l$.

%

\end{document}